\documentclass[a4paper,reqno]{amsart}

\usepackage[utf8]{inputenc}
\usepackage{amssymb}
\usepackage{enumitem}
\usepackage{color}
\usepackage{mathrsfs}
\usepackage{mathtools}
\setlist[enumerate]{label=\emph{(\roman*)}}

\usepackage{hyperref}

\newtheorem{theorem}{Theorem}[section]
\newtheorem{corollary}[theorem]{Corollary}
\newtheorem{lemma}[theorem]{Lemma}
\newtheorem{proposition}[theorem]{Proposition}

\theoremstyle{definition}
\newtheorem{definition}[theorem]{Definition}
\newtheorem{remark}[theorem]{Remark}
\newtheorem{assumption}[theorem]{Assumption}
\numberwithin{equation}{section}

\newcommand\R{\mathbb{R}}
\newcommand\e{\epsilon}

\newcommand\T{\mathbb{T}}
\newcommand{\la}{\langle}
\newcommand{\ra}{\rangle}

\begin{document}

\parindent=0pt

\title[Dynamics for helical vortex filament]
{Long time dynamics for helical vortex filament in Euler flows}

\author[D. Guo]{Dengjun Guo}
\address{School of Mathematical Sciences,
University of Science and Technology of China, Hefei 230026, Anhui, China}
\email{guodeng@mail.ustc.edu.cn}
\author[L. Zhao]{Lifeng Zhao}
\address{School of Mathematical Sciences,
University of Science and Technology of China, Hefei 230026, Anhui, China}
\email{zhaolf@ustc.edu.cn}
\email{}

\thanks{L. Zhao is supported by NSFC Grant of China No. 12271497, No. 12341102 and the National Key Research and Development Program of China No. 2020YFA0713100.
}
\thanks{\textbf{Data Availability Statements:} Data sharing not applicable to this article as no datasets were generated or analysed during the current study.}
\begin{abstract}
We consider the three-dimensional incompressible Euler equation
\begin{equation*}\left\{\begin{aligned}
&\partial_t \Omega+U \cdot \nabla \Omega-\Omega\cdot \nabla U=0  \\
&\Omega(x,0)=\Omega_0(x)
\end{aligned}\right.
\end{equation*}
under the assumption that $\Omega^z$ is helical and in the absence of vorticity stretching. Assuming that the initial vorticity $\Omega_0$ is primarily concentrated within an $\e$ neighborhood of a helix $\Gamma_0$, we prove that its solution $\Omega(\cdot,t)$ remain concentrated near a helix $\Gamma(t)$ for any $t \in [0,T)$,  where $\Gamma(t)$ can be interpreted as $\Gamma_0$ rotating around the $x_3$ axis with a speed $V=C\log \frac{1}{\e}+O(1)$. It should be emphasized that the dynamics for the helical vortex filament are exhibited on the time interval $[0,T)$, which is longer than $\left[0, \frac{T}{\log\frac{1}{\e}}\right)$.
\end{abstract}

\maketitle

\section{Introduction}

\quad Fix a time $T>0$, we consider the three-dimensional incompressible Euler equation in $\R^3 \times [0,T)$,
\begin{equation}\label{eq 3euler}\left\{\begin{aligned}
&\partial_t \Omega+U \cdot \nabla \Omega-\Omega\cdot \nabla U=0 \\
&\Omega(x,0)=\Omega_0(x).
\end{aligned}\right.
\end{equation}
The velocity $U$ can be recovered from the vorticity $\Omega$ by the well-known Biot-Savart law
\begin{equation}\label{eq biot savart law}U(x)=-\frac{1}{4\pi}\int_{\R^3}\frac{x-y}{|x-y|^3}\wedge \Omega(y)\,dy.\end{equation}
We are interested in the case when the initial vorticity is concentrated within an $\e$ neighborhood of a smooth curve $\Gamma_0$. Such solutions are usually called vortex filaments, and the their evolution has been formally described by Kelvin \cite{Kelvin}, Levi-civita \cite{Levi} and Da Rios \cite{Da1}. Loosely speaking, its solution $\Omega(\cdot,t)$ would remain concentrated near a smooth curve $\Gamma(s,t)$, which satisfies
\begin{equation}\label{eq sec1.1 A1}
\Gamma_t=(C_1 \log \e +C_2) \Gamma_s \wedge \Gamma_{ss}.
\end{equation}
The rigorous justification of this phenomenon is known as the vortex filament problem (conjecture). The only known result for general $\Gamma_0$ is provided by Jerrad and Seis \cite{Jer}, who proved that if $\Omega(\cdot,t)$ remain concentrated near a curve $\Gamma_1(t)$, then $\Gamma_1(t)$ must approximate $\Gamma(t)$ for at least $t \lesssim \frac{1}{\log{\frac{1}{\e}}}$. Up to now, the vortex filaments conjecture is still open for general $\Gamma_0$.

\quad However, there are numerous results available for specific choices of $\Gamma_0$. The vortex filament conjecture for straight vortex filament with $\Gamma_0(s)=(0,0,s)$ has been rigorously proven in the context of Euler equation with planar symmetry. Butta and Marchioro proved in \cite{MaP1} that if the initial vorticity is concentrated in an $\e$ neighborhood of $m$ disjoint straight vortex filaments, then this concentration will persist at least for $0\le t \lesssim\log \frac{1}{\e}$. For more results about straight vortex filaments, we refer to \cite{Mar}, \cite{MaP11}, \cite{MaP}, \cite{MaP10}, \cite{Tur}. Similarly, the vortex filament conjecture for vortex ring has also been rigorously justified for the axisymmetric Euler equation without swirl. In \cite{Ben}, Benedetto, Caglioti and Marchioro showed that when the initial vorticity is concentrated in an $\e$ neighborhood of $\Gamma_0(s)=(R\cos s, -R\sin s,0)$, then its solution would concentrated near $\Gamma(t)=(R\cos s,-R\sin s,Ct\log \frac{1}{\e})$ for any $t \in \left[0, \frac{T}{\log \e}\right)$. In \cite{MaP12}, Butta, Cavallaro and Marchioro extend this result to $m$ disjoint vortex rings. Moreover, the time range is extend to the entire interval $[0,T) $ in \cite{Ma}, under the assumption that $R$ is significantly large with respect to $\e$. Turning to helical vortex filament, which initially concentrated near $\Gamma_0(x)=(\cos s, -\sin s, s)$, Adebiyi \cite{Ade} constructed stable solutions for helical Euler equation with constant swirl. Dávila,del Pino, Musso and Wei \cite{Wei} constructed solutions that rotates around $x_3$ axis with a constant speed and does not change form in the context of helical Euler equation with vanishing swirl. In \cite{Cao}, Cao and Wan constructed multiple traveling-rotating helical vortices. See also \cite{Cao2}, \cite{Cao3} and \cite{Gue} for the construction of helical vortex filaments with specific choices of initial vorticity. Recently, Donati, Lacave and Miot consider helical vortex filaments for more general initial vorticity, they proved in \cite{Miot} that any helical solution initially concentrated around
helices of pairwise distinct radii remains concentrated close to filaments, specifically for any $t\in \left[0, \frac{T}{\log \frac{1}{\e}}\right)$.

\quad The main purpose of this paper is to study the dynamics for both straight and helical vortex filament for $t \in [0,T)$ and for more general initial vorticity, in the context of Euler equation with helical symmetry. To achieve this, we first review some known results for the Euler equation with helical symmetry. Regarding helical solutions, in \cite{Dut} Dutrifoy established the global well-posedness of smooth solutions in bounded helical domains. A key observation is that the vorticity $\Omega=\left( \Omega^x, \Omega^y, \Omega^z \right)$ can be rewritten as
 $$
 \Omega(x,t)=\xi(x)\Omega^z(x,t),
 $$
 where $\xi(x):=(x_2,-x_1,1)$ and $\Omega^z$ satisfies the three-dimensional helical Euler equation
\begin{equation}\label{eq 3euler z}
\partial_t \Omega^z+U \cdot \nabla \Omega^z=0.
\end{equation}

Setting $w(x_1,x_2)=\Omega^z(x_1,x_2,0)$, Ettinger and Titi \cite{ET} reduced the three-dimensional helical Euler equation to a two-dimensional problem:
    \begin{equation*}
\partial_t w+\nabla^{\perp}L_H^{-1}w \cdot \nabla w=0,
\end{equation*}

     where $L_H$ is a specific elliptic operator. They established the global existence and uniqueness of weak solutions to this two-dimensional problem in bounded helical domains. In the whole space $\R^3$, the global existence of weak solutions has been proved by Bronzi, Lopes and Lopes \cite{BLN} for initial vorticity belonging to $L^1\bigcap L^p(\R^3)$ with compact support. Subsequently, the assumption of the compact support was removed by Jiu, Li and Niu \cite{JLN}. In \cite{GZ1}, the authors further reduced the three-dimensional helical Euler equation to the two-dimensional helical Euler equation:
     \begin{equation}\label{eq 2euler}
\partial_t w +Hw\cdot \nabla w=0,
\end{equation}
where $Hw$ is a divergence-free velocity field defined as
\begin{equation*}
Hw(x)=\int_{\R^2} H(x,y)w(y)\,dy
\end{equation*}
and $H(x,y)$ is the modified Biot-Savart kernel described in Proposition \ref{prop 3 euler}. They established global well-posedness for weak solutions in $L^1_1 \bigcap L^{\infty}_1$. It is worth noting that, unlike the axisymmetric or planar Euler equations, the Biot-Savar law for helical Euler equations (see, e.g. \cite{Ade}, \cite{GZ1}, \cite{JLN}) is usually given by
\begin{equation*}
U(x):=-\frac{1}{4\pi}\int_{\R^3}\frac{x-y}{|x-y|^3}\wedge \Omega(y)\,dy+(0,0,W),
\end{equation*}
where $W\in \R$ is a constant.  In \cite{JLN}, Jiu, Li and Niu chose $$W=\frac{1}{4\pi^2}\int_{\R^2\times \T}\Omega_0^z(y)\,dy$$ to ensure the vanishing swirl condition. While in \cite{Ade}, Adebiyi construct stable helical vortex filament by setting $$W=-C_1\log\frac{1}{\e}-C_2+o(1)$$ to eliminate the influence of the rotation induced by the helix itself. Note that our main purpose of this paper is to prove vortex filament conjecture in helical setting. So we will consistently assume $W=0$ and focus on the helical Euler equation as outlined in \cite{GZ1}. It is worth mentioning that our method would remain valid for any $W \in \R$ that is independent of $\e$.

\quad Next we describe the vortex filament conjecture for helical Euler equation. As demonstrated in \cite{GZ1}, if $w(x,t)$ solves the two-dimensional helical Euler equation \eqref{eq 2euler}, then $\Omega(x,t):=(x_2,-x_1,1)w\left(R_{-x_3}(x_1,x_2),t\right)$ solves the three-dimensional Euler equation \eqref{eq 3euler}. Therefore, for the two-dimensional helical Euler equation \eqref{eq 2euler}, the vortex filament conjecture reads: if $w_{\e,0}(x)$ is concentrated near $\sigma=(\sigma_1,\sigma_2)$, then its solution $w_{\e}(x,t)$ will rotate uniformly around origin with a speed of $V=C_{1,\sigma}\log\frac{1}{\e}+C_{2,\sigma}$. The case $\sigma=0$ corresponds to straight vortex filament and $\sigma \neq 0$ corresponds to helical vortex filament. As we will see in Section $2$, the conservation of the second momentum
$$
M(t):=\int_{\R^2} |x|^2w(x,t)\,dx
$$
directly leads to the existence of straight vortex filament (Theorem \ref{thm sec2 A1}) and the convergence of the helical filament to straight filament when $\sigma \to 0$ (Corollary \ref{co sec2 A1}). Moreover, the evolution of $w_{\e,\sigma}(\cdot)$ for different values of $\sigma \neq 0$ can be treated similarly. Hence for simplicity, we will solely consider the case when $\sigma=(1,0)$ and assume
$$
\int_{\R^2} w_{\e,0}(x)\,dx =1.
$$

The main result is stated as follows (for a
precise definition of weak solutions and the relevant function spaces, please refer to Section 2).
\begin{theorem}\label{thm GZ main}
Given $T>0$, let $\,w_{\e}(x,t)$ be a weak solution to the two-dimensional helical Euler equation \eqref{eq 2euler} with initial data $w_{\e,0}(x) $. Assuming in addition that $w_{\e,0}(x) \in L^1_4 \bigcap L^{\infty}_4$ satisfies Assumption \ref{as initial data}, then for any $t \in [0,T)$ and $0<\alpha<1$, there exist $\e_0>0$, $p_{\e}^*(t) \in \R^2$ and a constant $C>0$ independent of $\e, \alpha$ such that
\begin{equation*}
\int_{|x-p_{\e}^*(t)|>\e^{1-\alpha}}\langle x \rangle^2 w_{\e}(x,t)\,dx \le \frac{C}{\alpha \log \frac{1}{\e}}
\end{equation*}
holds whenever $\e \le \e_0$. Furthermore, define $$p_{\e}(t)=\int_{\R^2} xw_{\e}(x,t)\,dx,$$ then for any $t\in [0,T)$, there hold
\begin{equation*}\begin{aligned}
|p_{\e}^*(t)-p_{\e}(t)|&=O\left(\frac{1}{\log\frac{1}{\e}}\right), \\
|p_{\e}(t)|&=1+O\left(\frac{1}{\log\frac{1}{\e}}\right)
\end{aligned}\end{equation*}
and
\begin{equation*}
p_{\e}(t)=\left( \cos \left(V_{\e}t \right),\sin \left(V_{\e}t \right) \right) +O\left( t  \right),
\end{equation*}
where
$$
V_{\e}=-\frac{\sqrt2}{8\pi}\log\frac{1}{\e}.
$$
\end{theorem}

$$$$
\begin{remark}$\,$

\begin{enumerate}

    \item   The conditions stated in Assumption \ref{as initial data} are natural and can be verified directly if the initial data has the form $$w_{\e,0}(x)=\frac{1}{\e^2}\eta(\frac{x-(1,0)}{\e}),$$ where $\eta$ is a non-negative function with sufficient decay at infinity and satisfies $$\int_{\R^2} \eta(x)\,dx=1, \,\,\,\,\,\, \int_{\R^2} x\eta(x)\,dx=0.$$
   \item To our best knowledge, except for straight vortex filament, the evolution for general initial vorticity has only been derived for a limited time interval $t\in \left[0, \frac{T}{\log\frac{1}{\e}}\right)$. The long time dynamic is even unknown for vortex rings. Our results demonstrate that the helical vortex filament rotates around the origin with a speed $V=V_{\e}+O(1)$ for $t\in [0, T)$, which can be regarded as a weak version of the vortex filament conjecture \eqref{eq sec1.1 A1} in helical setting.
\end{enumerate}
\end{remark}
Our method is different from previous work on vortex rings and helical vortex filaments, which rely on a detailed analysis of the Lagrangian trajectories to Euler equations and require that the vorticity has compact support. In this article, we have developed a new method to investigate the helical vortex filament. Notably, we do not use the Lagrangian approach and make no assumptions about the support of the initial vorticity. As we will see in Section 5, this new method is well-suited for studying the long-time behavior of helical vortex filament.
Our approach involves demonstrating that the energy for \eqref{eq 3euler z}  behaves similarly to the $2D$ case,
$$
E(t) \approx \int_{\R^2} \int_{\R^2} \log \frac{1}{|x-y|}w(x,t)w(y,t)\,dx\,dy.
$$
To accomplish this, we first establish more delicate estimates for the Biot-Savart kernel than those previously achieved in \cite{GZ1}. By combining the estimates for energy, momentum and the distance function defined in Section $4$, we are able to show that $w_{\e}(\cdot,t)$ remains concentrated near its center of gravity $p_{\e}(t)$ for any time $t \ge 0$, in the sense that
\begin{equation*}
 \int_{|x-p^*_{\e}(t)|\ge \e^{1-\alpha}} \la x\ra w_{\e}(x,t)\,dx \lesssim \frac{1}{\log \frac{1}{\e}}
 \end{equation*}
 for any $0<\alpha <1$. Next we control \begin{equation}\label{eq sec1 B9}\int_{|x-p^*_{\e}(t)|\ge \e^{1-\alpha}}\la x\ra^2 w_{\e}(x,t)\,dx \end{equation} with the assistance of a crucial estimate for
\begin{equation*}
M^*(t):=\int_{\R^2} |x|^2 \eta(|x|)w_{\e}(x,t)\,dx,
\end{equation*}
where $0\le \eta \le 1$ is a smooth positive cut-off function satisfies
\begin{equation}\eta(r)=\begin{cases}
0 \quad r \le 10 \\
1 \quad r>20.
\end{cases}\end{equation}
 Subsequently, through a functional equality linking the distance function, the second momentum, and the center of gravity, the control for \eqref{eq sec1 B9} directly implies that $|p_{\e}(t)|=1+O\left(\frac{1}{\log \frac{1}{\e}}\right)$. Finally, we show that the the center of gravity, $p_{\e}(t)$, satisfies the differential equation $$\frac{dp_{\e}}{dt}=V_{\e}p_{\e}^{\perp}+O(1),$$ which serves as the conclusive step in establishing the long-time dynamics for helical vortex filament.

\quad Our article is organized as follows: In Section $2$, we provide a comprehensive overview of helical Euler equations. As a straightforward consequence of Theorem \ref{thm main}, we prove the existence of straight vortex filament and the convergence of helical vortex filament to the straight vortex filament when $\sigma \to 0$. Section 3 focuses on obtaining delicate estimates for the stream function $\left( -\Delta\right)^{-1}\Omega$ and the velocity $Hw$. In Section $4$, we introduce and analyze several crucial quantities. Finally, in Section $5$, we prove Theorem \ref{thm GZ main} by using the crucial estimates obtained in previous section.

\section{Mathematical preliminaries and main result. }
The main purpose of this section is to fix notations and review some known results for helical Euler equations. For simplicity of presentation, we usually refer a point $x=(x_1,x_2,x_3) \in \R^3$ or $\R^2\times \T$ to $x=(x',x_3)$. We adopt the Japanese bracket notation, defined as $\la x \ra =\sqrt{1+|x|^2}$. The notation $a \lesssim b$ means that there exists a constant $C>0$ independent of $x,y,\e$ (but possibly dependent on $T$) such that $a \le Cb$. Additionally, $a=O(b)$ denotes that $|a| \lesssim |b|$.
\subsection{Helical functions and vector fields. }

\begin{definition}[helical function]\label{def helical function}
A function $f: \R^3 \to \R$ is called helical if for all $\theta \in \R$ and almost every $x\in \R^3$,
\begin{equation*}
f(S_{\theta}x)=f(x),
\end{equation*}
where
\begin{equation*}
S_{\theta}x:=R_{\theta}x+\begin{pmatrix} 0 \\ 0 \\ \theta \end{pmatrix}
\end{equation*}
and
\begin{equation*}
R_{\theta}:= \begin{pmatrix}  \cos{\theta} & \sin{\theta} & 0 \\ -\sin{\theta} & \cos{\theta} & 0 \\ 0 & 0 &1 \end{pmatrix}.
\end{equation*}
\end{definition}

\begin{definition}[helical vector field]\label{def helical vector field}
A vector field $u: \R^3 \to \R^3$ is called helical if for all $\theta \in \R$,
\begin{equation*}
R_{-\theta}u(S_{\theta}(x))=u(x)
\end{equation*}
for almost every $x \in \R^3$.

\end{definition}
 Assume $f$ is a continuous helical function and $u$ is a continuous helical vector field. With a slight abuse of notation, we define $R_{\theta}x=\begin{pmatrix}  \cos{\theta} & \sin{\theta}  \\ -\sin{\theta} & \cos{\theta} \end{pmatrix} \begin{pmatrix} x_1 \\ x_2 \end{pmatrix}.$ By introducing $\mathfrak{R}=\begin{pmatrix}  0 & 1  \\ -1 & 0 \end{pmatrix}$, we observe that
 $$R_ax=e^{a\mathfrak{R}}x$$
and $$-\mathfrak{R}x=(-x_2,x_1):=x^{\perp}.$$

 From the definitions above, it follows that
\begin{equation}\label{eq helical function}
f(x_1,x_2,x_3)=f(R_{-x_3}(x_1,x_2),0)
\end{equation}
and
\begin{equation}\label{eq helical vector}
u(x_1,x_2,x_3)=R_{x_3}u(R_{-x_3}(x_1,x_2),0).
\end{equation}

Note that if $f$ and $u$ are only measurable function (vector field), then $f(x',0)$ and $u(x',0)$ are not well-defined, thus \eqref{eq helical function} and \eqref{eq helical vector} do not make sense. However, we have the following:
\begin{lemma}[\cite{GZ1}]\label{le sec2 C7}
Let $f$ be a locally bounded helical function and $u$ be a locally bounded helical vector field. Define
$$
\tilde{f}(x')=\frac{1}{2\pi}\int_{-\pi}^{\pi} f(R_{a}x',a) \,da
$$
and
$$
\tilde{u}(x')=\frac{1}{2\pi}\int_{-\pi}^{\pi} R_{-a}u(R_{a}x',a) \,da.
$$
Then for almost every $x\in \R^3$, there hold
\begin{equation*}
f(x',x_3)=\tilde{f}(R_{-x_3}x')
\end{equation*}
and
\begin{equation*}
u(x',x_3)=R_{x_3}\tilde{u}(R_{-x_3}x').
\end{equation*}

\end{lemma}
$$$$

Motivated by \eqref{eq helical function}, \eqref{eq helical vector}, by slight abuse of notations, we may refer to $\tilde{f}(x')$ as $f(x',0)$, and refer to $\tilde{u}(x')$ as $u(x',0)$.

Next, we summarize the following results for Euler equations with helical symmetry.

\begin{proposition}[\cite{GZ1}\,Two-dimensional helical Euler equation]\label{prop 3 euler}
Let $\Omega^z$ be a smooth helical solution to the three-dimensional helical Euler equation \eqref{eq 3euler z}. We define the two-dimensional modified Biot-Savart kernel
\begin{equation*}
K(x,y)=-\frac{1}{4\pi} \int_{\R} \frac{(x_1,x_2,0)-\left(R_{a}(y),a\right)}{|(x_1,x_2,0)-\left(R_{a}(y),a\right)|^3} \wedge \xi(\left(R_{a}(y),a\right))\,da.
\end{equation*}
Set $w(x_1,x_2)=\Omega^z(x_1,x_2,0)$,
\begin{equation*}
U(x_1,x_2)=\int_{\R^2} K(x,y)w(y)\,dy
\end{equation*}
and
\begin{equation*}
Hw:=H_1w+H_2w=(U^1,U^2)+(-x_2,x_1)U^3.
\end{equation*}
Then $w$ satisfies the two-dimensional helical Euler equation
\begin{equation*}
\partial_t w+Hw\cdot \nabla w=0
\end{equation*}
with $\nabla \cdot Hw=0$.

\end{proposition}

\subsection{Function spaces and the definition of weak solutions}
In this subsection, we introduce functional spaces and the definition of weak solutions.
The three-dimensional weighted $L^p_m(\R^2 \times \T)$ norm is defined by
\begin{equation*}
\|f\|_{L^p_m(\R^2\times \T)}:=\int_{\R^2\times \T}(1+y_1^2+y_2^2)^{\frac{pm}{2}} |f(y_1,y_2,y_3)|^p\,dy^{\frac{1}{p}}
\end{equation*}
for $1\le p< \infty$ and
$$\|f\|_{L^{\infty}_m(\R^2\times \T)}:=\sup_{x\in \R^3}(1+x_1^2+x_2^2)^{\frac{m}{2}} |f(x_1,x_2,x_3)|.$$
Similarly, the two-dimensional weighted $L^p_m(\R^2)$ norm is defined by
\begin{equation*}
\|g\|_{L^p_m(\R^2)}:=\int_{\R^2}\la y \ra^{pm} |g(y_1,y_2)|^p\,dy^{\frac{1}{p}}
\end{equation*}
for $1\le p< \infty$ and
$$\|f\|_{L^{\infty}_m(\R^2)}:=\sup_{x\in \R^2}\la x \ra^{m} |f(x_1,x_2)|.$$
The norms $L^p_m(\R^2 \times \T)$ and $L^p_m(\R^2)$ are related as follows:
\begin{lemma}\label{le norm of R^2 and R^3}
Let $f \in L^p_m(\R^2)$ and define $F: \R^3 \to \R$ by
\begin{equation*}
F(x',x_3):=f(R_{-x_3}x').
\end{equation*}
Then for any $p \in [1,+\infty]$,
\begin{equation*}
\|f\|_{L^p_m(\R^2)}=(2\pi)^{-1/p}\|F\|_{L^p_m(\R^2 \times \mathbb{T})}.
\end{equation*}
\end{lemma}

We will now formulate the weak problem of the two-dimensional helical Euler equation \eqref{eq 2euler} and the three-dimensional helical Euler equation \eqref{eq 3euler z}.
\begin{definition}
(Weak solutions to the three-dimensional helical Euler equation.) We say that $\Omega^z(x,t)$ is a weak solution to \eqref{eq 3euler z} if
\begin{equation*}
\int_{\R^3} \Omega^z(x,t)\phi(x,t) \,dx-\int_{\R^3} \Omega^z(x,0)\phi(x,0)\,dx=\int_0^t \int_{\R^3} \Omega (\partial_t \phi +U\cdot \nabla \phi) \,dxds.
\end{equation*}
for all $t\in [0,T]$ and for all test function $\phi \in C_c^{\infty}(\R^3\times [0,+\infty))$.
\end{definition}
Similarly, we define
\begin{definition}
(Weak solutions to the two-dimensional helical Euler equation.) We say that $w(x,t)$ is a weak solution to \eqref{eq 2euler} if
\begin{equation*}
\int_{\R^2} w(x,t)\phi(x,t) \,dx-\int_{\R^2} w(x,0)\phi(x,0)\,dx=\int_0^t \int_{\R^2} w (\partial_t \phi +Hw\cdot \nabla \phi) \,dxds.
\end{equation*}
for all $t\in [0,T]$ and for all test function $\phi \in C_c^{\infty}(\R^2\times [0,+\infty))$.
\end{definition}

\begin{theorem}[\cite{GZ1}]\label{thm main}
For any positive integer $m$, the helical Euler equation \eqref{eq 2euler} and \eqref{eq 3euler z} are globally well-posed in $L^1_m\bigcap L^{\infty}_m(\R^2)$ and $L^1_m\bigcap L^{\infty}_m(\R^2\times \T)$, respectively. Moreover, whenever $m \ge 2$, the second momentum
$$
M_1(t):=\int_{\R^2\times\T} (x_1^2+x_2^2)\Omega^z(x,t)\,dx
$$
and the (pseudo) energy
$$
E(t):=\int_{\R^2\times \T} (-\Delta_{\R^2\times \T})^{-1}\Omega \cdot \Omega \,dx
$$
are conserved, where $\Omega(x,t):=(x_2,-x_1,1)\Omega^z(x,t)$.
\end{theorem}
The conservation of the second momentum leads directly to the existence of straight vortex filament.
\begin{theorem}\label{thm sec2 A1}
Let $w_{\e}(x,t)$ satisfy the two-dimensional helical Euler equation \eqref{eq 2euler}. Assuming that the initial vorticity $w_{\e,0}(x)$ is non-negative,
$$
\int_{\R^2} w_{\e,0}(x)\,dx=1
$$
and
$$
\int_{\R^2} |x|^2w_{\e,0}(x)\,dx \lesssim \e^2.
$$
Then for any $\alpha<1 $ and $t\ge0$, there holds
\begin{equation*}
\int_{|x|\ge \e^{\alpha}}w_{\e}(x,t)\,dx \lesssim \e^{2-2\alpha}.
\end{equation*}
In particular, $w_{\e}(\cdot,t)$ converges to the Dirac measure in distribution as $\e \to 0$.
\end{theorem}
\begin{proof}
Note that $\Omega_{\e}^z(x',x_3,t)=w_{\e}(R_{x_3}x',t)$. According to Theorem \ref{thm main}, we have
$$
\int_{\R^2} w_{\e}(x,t)\,dx=1
$$
and
$$
\int_{\R^2} |x|^2w_{\e}(x,t)\,dx \lesssim \e^2
$$
for any time $t \ge0$.
Therefore,
$$
\int_{|x|\ge \e^{\alpha}}w_{\e}(x,t)\,dx \lesssim \int_{|x|\ge \e^{\alpha}}\frac{|x|^2}{\e^{2\alpha}}w_{\e}(x,t)\,dx \lesssim \e^{2-2\alpha}.
$$
\end{proof}
As a consequence, we will show the convergence of the helical vortex filament to the straight vortex filament. We consider the initial vorticity $w_{\e,\sigma}(x,0)$ which satisfies
$$
\int_{\R^2} w_{\e,\sigma}(x,0)\,dx=1
$$
and
$$
\int_{\R^2} |x|^2w_{\e,\sigma}(x,0)\,dx \lesssim \e^2+|\sigma|^2.
$$
These conditions can be easily verified if $w_{\e,\sigma}$ has the form
$$
w_{\e,\sigma}(x,0)=\frac{1}{\e^2}\eta\left(\frac{x-\sigma}{\e}\right)
$$
for $\eta$ a non-negative function with sufficient decay at infinity and satisfies $$\int_{\R^2} \eta(x)\,dx =1.$$
\begin{corollary}\label{co sec2 A1}
Let $w_{\e,\sigma}(x,t)$ satisfy the two-dimensional helical Euler equation \eqref{eq 2euler}. Assuming that $w_{\e,\sigma}(x,0)$ satisfies the conditions stated above, then for any $\alpha<1 $ and $t\ge0$, there holds
\begin{equation*}
\int_{|x|\ge \left(\e^2+|\sigma|^2\right)^{\frac{\alpha}{2}}}w_{\e,\sigma}(x,t)\,dx \lesssim \left(\e^2+|\sigma|^2\right)^{1-\alpha}.
\end{equation*}
\end{corollary}
\begin{proof}
This follows directly from Theorem \ref{thm sec2 A1}.
\end{proof}
\section{Biot-Savart law}

\subsection{Stream function}
Recall that the energy
$$
E(t):=\int_{\R^2\times \T} (-\Delta_{\R^2\times \T})^{-1}\Omega \cdot \Omega \,dx
$$
is a conserved quantity, our main task in this section is to show that $E(t)$ behaves like $$\int_{\R^2}\int_{\R^2} \log\frac{1}{|x-y|} w(x)w(y)\,dx\,dy,$$
which is already known for $2D$ Euler equations and axisymmetric Euler equations without swirl. Our main theorem is,
\begin{theorem}\label{thm sec3 A0}
The energy $E(t)$ admits the following decomposition
\begin{equation*}
E(t)=\int_{\R^2}\int_{\R^2}\log \frac{1}{|x-y|}\mathbf{1}_{|x-y|\le 1}\langle x \rangle  w(x)w(y)\,dx\,dy+ O\left(\| w\|_{L^1_2(\R^2)}^2\right).
\end{equation*}
\end{theorem}
To prove this theorem, we start by considering the Green's function on $\R^2 \times \T$,
\begin{proposition}[{\cite{BLN}, \cite{GZ1}}]\label{prop stream function}
The Green's function for the Laplacian in $\R^3$ with $2\pi-$periodic boundary condition in the $x_3-$direction is given by
\begin{equation}\label{eq sec2.111 A1}
G(x)=\frac{1}{4\pi^2}\log\frac{1}{|\tilde{x}|}+\frac{1}{2\pi^2}\sum_{n=1}^{\infty}K_0(n|\tilde{x}|)\cos(nx_3)
\end{equation}
for all $x \in \R^2\times \T$ with $\tilde{x}\neq 0$, where
\begin{equation*}
K_0(z):=\int_0^{\infty}\frac{\cos(t)}{\sqrt{z^2+t^2}}\,dt.
\end{equation*}
More precisely, the stream function $$\psi(x):=\int_{\R^2\times \T} G(x-y)f(y)\,dy$$ is a periodic function in $x_3$ and satisfies $$-\Delta \psi=f.$$
\end{proposition}
Next, we analyze the behavior of $G$ near $|x'|=0$.

\begin{lemma}\label{le sec3 G}
For any $x\in \R^2 \times \T$ with $x'\neq 0$, the Green's function admits the following decomposition (here $\gamma$ denote the Euler constant)
\begin{equation}\label{eq Green's function behavior}
G(x)=\frac{1}{4\pi|x|}+\frac{\gamma-\log (4\pi)}{4\pi^2}+\mathcal{A}_G(x),
\end{equation}
where the remainder term $\mathcal{A}_G$ is relatively small near $|x'|=0$ in the sense that
\begin{equation*}
|\mathcal{A}_G(x)| \lesssim |x|
\end{equation*}
and
\begin{equation*}
|\nabla^n\mathcal{A}_G(x)| \lesssim 1
\end{equation*}
for any $n \in \mathbb{Z}^+$.
\end{lemma}
\begin{proof}
As shown in \cite{BLN} and \cite{GZ1}, the series in \eqref{eq sec2.111 A1} can be rewritten as
\begin{equation*}\begin{aligned}
\sum_{n=1}^{\infty}K_{0}(n|x'|)\cos(nx_3)&=\frac{1}{2}\left(\gamma+\log\left(\frac{1}{4\pi}\right)+\log(|x'|)  \right)+\frac{\pi}{2|x|}\\
+&\frac{\pi}{2}\sum_{m=1}^{\infty}\left[ \frac{1}{\sqrt{|x'|^2+(x_3-2m\pi)^2}}-\frac{1}{2m\pi} \right] \\
+&\frac{\pi}{2}\sum_{m=1}^{\infty}\left[ \frac{1}{\sqrt{|x'|^2+(x_3+2m\pi)^2}}-\frac{1}{2m\pi} \right].
\end{aligned}\end{equation*}
Therefore, \eqref{eq Green's function behavior} follows with
\begin{equation*}\begin{aligned}
\mathcal{A}_G(x)&=\frac{\pi}{2}\sum_{m=1}^{\infty}\left[ \frac{1}{\sqrt{|x'|^2+(x_3-2m\pi)^2}}-\frac{1}{2m\pi} \right] \\
+&\frac{\pi}{2}\sum_{m=1}^{\infty}\left[ \frac{1}{\sqrt{|x'|^2+(x_3+2m\pi)^2}}-\frac{1}{2m\pi} \right].
\end{aligned}\end{equation*}
Recall that $x_3 \in [-\pi,\pi]$, which implies that $|x_3 \pm 2m\pi|\approx 2m\pi$ and $|x_3| \lesssim 2m\pi$. A direct calculation yield
$$
|\mathcal{A}_G(x)| \lesssim \sum_{m=1}^{\infty} \frac{|x|}{m^2}\lesssim |x|
$$
and
$$
|\nabla^n\mathcal{A}_G(x)| \lesssim \sum_{m=1}^{\infty} \frac{1}{m^{n+1}} \lesssim 1.
$$
\end{proof}
As a consequence, we see that (set $\xi(x)=(x_2,-x_1,1)$)
\begin{equation}
E(t)=\int_{\R^2\times \T}\int_{\R^2\times \T} \frac{\xi(x)\cdot \xi(y)}{4\pi|x-y|} \Omega^z(x)\Omega^z(y)\,dx\,dy+\mathcal{A}_{E,1},
\end{equation}
where
$$
|\mathcal{A}_{E,1}| \lesssim \| \Omega^z\|_{L^1_2(\R^2\times \T)}^2.
$$
Note that $|\xi(x)-\xi(y)|\le |x-y|$ and $\Omega^z(x)=w(R_{-x_3}(x_1,x_2))$, we finally get
\begin{equation}\label{eq sec3 D2}\begin{aligned}
E(t)&=\int_{\R^2\times \T}\int_{\R^2\times \T} \frac{1+x_1^2+x_2^2}{4\pi|x-y|} \Omega^z(x)\Omega^z(y)\,dx\,dy+\mathcal{A}_{E,2} \\
&=\int_{\R^2}\int_{\R^2}  K_1(x,y)(1+x_1^2+x_2^2)w(x)w(y)\,dx\,dy+\mathcal{A}_{E,2},
\end{aligned}\end{equation}
where
\begin{equation*}
\mathcal{A}_{E,2}=\int_{\R^2\times \T}\int_{\R^2\times \T} \frac{\xi(x)\cdot \left(\xi(y)-\xi(x)\right)}{4\pi|x-y|} \Omega^z(x)\Omega^z(y)\,dx\,dy+\mathcal{A}_{E,1},
\end{equation*}
is also bounded by $\| \Omega^z\|_{L^1_2(\R^2\times \T)}^2$ and
$$K_1(x,y)=\frac{1}{4\pi}\int_{-\pi}^{\pi}\int_{-\pi}^{\pi}  \frac{1}{\left( |R_{a}x-R_by|^2+|a-b|^2 \right)^{1/2}} \,da\,db.$$
To estimate $K_1(x,y)$, we proceed as follows:
\begin{equation*}\begin{aligned}
K_1(x,y)&=\frac{1}{2}\int_{-\pi}^{\pi}\int_{-\pi}^{\pi}  \frac{1}{\left( |x-R_{b-a}y|^2+|b-a|^2 \right)^{1/2}} \,db\,da \\
&=\frac{1}{4\pi}\int_{-\pi}^{\pi}\int_{-\pi-a}^{\pi-a}  \frac{1}{\left( |x-R_{b}y|^2+|b|^2 \right)^{1/2}} \,db\,da \\
&=\frac{1}{4\pi}\int_{-2\pi}^{2\pi}\int \mathbf{1}_{-\pi \le a+b \le \pi, \, -\pi \le a \le \pi} \,da \frac{1}{\left( |x-R_{b}y|^2+|b|^2 \right)^{1/2}} \,db.
\end{aligned}\end{equation*}
Observing that
$$
\int \mathbf{1}_{-\pi \le a+b \le \pi, \, -\pi \le a \le \pi} \,da=2\pi-|b|,
$$
we conclude that
\begin{equation*}\begin{aligned}
K_1(x,y)&=\frac{1}{4\pi}\int_{-2\pi}^{2\pi} \frac{2\pi-|b|}{\left( |x-R_{b}y|^2+|b|^2 \right)^{1/2}} \,db \\
&=\frac{1}{2}\int_{-\pi}^{\pi} \frac{1}{\left( |x-R_{a}y|^2+|a|^2 \right)^{1/2}} \,da+O(1).
\end{aligned}\end{equation*}

\begin{lemma}\label{le sec3 A1}There exists a constant $\delta_1$  independent of $x,y$ such that
\begin{equation}\label{eq sec3 D0}
K_1(x,y)=\frac{1}{\sqrt{1+|x|^2}}\log\frac{1}{|x-y|}+O(1)
\end{equation}
when $|x-y|\le \delta_1$ and
\begin{equation}\label{eq sec3 D1}
K_1(x,y) \lesssim 1
\end{equation}
when $|x-y|\ge \delta_1$.
\end{lemma}
\begin{proof}
We fixed $\delta_0$ a small constant independent of $x,y$, the choice of $\delta_0$ will be specified later. Inequality \eqref{eq sec3 D1} follows directly from the estimates
\begin{equation}\label{eq sec3 E0}
\int_0^{\infty} \frac{1}{|x-R_a(y)|^2+a^2} \,da \lesssim 1+\frac{1}{|x-y|},
\end{equation}
which has been proved in \cite{GZ1}. So it remains to estimate
$$
I(x,y)=\int_{-\pi}^{\pi} \frac{1}{\left( |x-R_{a}y|^2+|a|^2 \right)^{1/2}} \,da
$$
in the region $|x-y|\le \delta_1$. It is worth noting that $I(x,y)=I(y,x)$, so it suffices to consider the region $$\left\{ (x,y)\in \R^2\times \R^2 \bigg| x^{\perp}\cdot y \ge 0 \right\}.$$

By choosing $\delta_0$ small enough such that $$|a^2-2+2\cos a|\le 1-\cos a$$ for all $|a|\le \delta_0$, we obtain
\begin{equation*}\begin{aligned}
I(x,y)&=\int_{-\delta_0}^{\delta_0} \frac{1}{\left( |x-R_{a}y|^2+2-2\cos a+O(a^4) \right)^{1/2}} \,da+O(1)\\
&=\int_{-\delta_0}^{\delta_0} \frac{1}{\left( |x-R_{a}y|^2+2-2\cos a \right)^{1/2}}\,da+\int_{-\delta_0}^{\delta_0} \frac{O(a^2)}{\left( |x-R_{a}y|^2+2-2\cos a \right)^{1/2}} \,da+O(1) \\
&=\int_{-\delta_0}^{\delta_0} \frac{1}{\left( |x-R_{a}y|^2+2-2\cos a \right)^{1/2}}\,da +O(1) \\
&=\int_{-\pi}^{\pi} \frac{1}{\left( |x-R_{a}y|^2+2-2\cos a \right)^{1/2}}\,da +O(1).
\end{aligned}\end{equation*}
Next, a direct calculation shows
$$
|x-R_ay|^2+2-2\cos a=|x|^2+|y|^2+2-2\left( (1+x\cdot y) \cos a+x^{\perp}\cdot y \sin a \right).
$$
Now, by choosing $\delta_1$ small enough such that $1+x\cdot y \ge \frac12$ whenever $|x-y|\le \delta_1$, there holds
$$
|x-R_ay|^2+2-2\cos a=|x|^2+|y|^2+2-2\mathfrak{g}(x,y) \cos\left(a+\phi(x,y)-\frac{\pi}{2}\right),
$$
where
$$
\mathfrak{g}(x,y):= \sqrt{(1+x\cdot y)^2+(x^{\perp}\cdot y)^2}
$$
and
$$
\phi(x,y)=\arctan \frac{1+x\cdot y}{x^{\perp}\cdot y}.
$$
Thus, $0 \le \phi(x,y) \le \frac{\pi}{2}$ since $1+x\cdot y \ge \frac12$ and $x^{\perp} \cdot y \ge 0$. Consequently,
\begin{equation*}\begin{aligned}
I(x,y)&=\int_{-\pi}^{\pi} \frac{1}{\left( |x|^2+|y|^2+2-2\mathfrak{g}(x,y) \cos\left(a+\phi(x,y)-\frac{\pi}{2}\right) \right)^{1/2}}\,da +O(1) \\
&=\frac{1}{\sqrt{2\mathfrak{g}(x,y)}}\int_{-\frac{3\pi}{2}+\phi(x,y)}^{\frac{\pi}{2}+\phi(x,y)} \frac{1}{\sqrt{\frac{\mathfrak{g}_1^2(x,y)}{2}+1-\cos a}}\,da+O(1),
\end{aligned}\end{equation*}
where
$$
\frac{\mathfrak{g}_1^2(x,y)}{2}:=\frac{4|x-y|^2+(|x|^2-|y|^2)^2}{2\mathfrak{g}(x,y)(|x|^2+|y|^2+2+2\mathfrak{g}(x,y))} \ge0.
$$
Note that $0\le \phi(x,y) \le \frac{\pi}{2}$ and $\mathfrak{g}(x,y)\ge \frac12$,
\begin{equation*}\begin{aligned}
I(x,y)=\frac{1}{\sqrt{2\mathfrak{g}(x,y)}}\int_{-\delta_0}^{\delta_0} \frac{1}{\sqrt{\frac{\mathfrak{g}_1^2(x,y)}{2}+1-\cos a}}\,da+O(1).
\end{aligned}\end{equation*}
Choosing $\delta_0$ small enough such that $\cos a-1+\frac{a^2}{2} \le \frac{1-\cos a}{2}$ for all $|a|\le \delta_0$, we obtain
\begin{equation*}\begin{aligned}
I(x,y)&=\frac{1}{\sqrt{2\mathfrak{g}(x,y)}}\int_{-\delta_0}^{\delta_0} \frac{1}{\sqrt{\frac{\mathfrak{g}_1^2(x,y)}{2}+\frac{a^2}{2}}}\,da+O(1)\\
&=\frac{2}{\sqrt{\mathfrak{g}(x,y)}}\int_{0}^{1} \frac{1}{\sqrt{\mathfrak{g}_1^2(x,y)+a^2}}\,da+O(1)\\
&=\frac{2}{\sqrt{\mathfrak{g}(x,y)}}\int_{0}^{\frac{1}{\mathfrak{g}_1(x,y)}} \frac{1}{\sqrt{1+a^2}}\,da+O(1).
\end{aligned}\end{equation*}
The fact $\mathfrak{g}_1(x,y)\lesssim |x-y|$ allows us to choose $\delta_1$ small enough such that $\mathfrak{g}_1(x,y) \le \frac{1}{100}$. Consequently, we have
$$
I(x,y)=\frac{2}{\sqrt{\mathfrak{g}(x,y)}} \log \frac{1}{\mathfrak{g}_1(x,y)}+O(1).
$$
A direct calculation yields
$$
\log \frac{1}{\mathfrak{g}_1(x,y)}=\log \frac{1}{|x-y|}+O(|x|+|y|+1)
$$
and since $|x-y|\le \delta_1$ is small, we also have
$$
\frac{2}{\sqrt{\mathfrak{g}(x,y)}}=\frac{2}{\sqrt{\mathfrak{g}(x,x)}}+O(1)=\frac{2}{\sqrt{1+|x|^2}}+O(1).
$$
Therefore, we obtain
$$
I(x,y)=\frac{2}{\sqrt{1+|x|^2}}\log\frac{1}{|x-y|}+O(1)
$$
when $|x-y|\le \delta_1$, which completes the proof of the Lemma.
\end{proof}
It follows from \eqref{eq sec3 D2} and Lemma \ref{le sec3 A1} that
\begin{equation*}
E(t)=\int_{\R^2}\int_{\R^2}\log \frac{1}{|x-y|}\mathbf{1}_{|x-y|\le \delta_1} \langle x \rangle w(x)w(y)\,dx\,dy+ O\left(\| \Omega^z\|_{L^1_2(\R^2\times \T)}^2\right).
\end{equation*}
Noting that $\log \frac{1}{|x-y|}\approx 1$ whenever $\delta_1 \le |x-y|\le 1$, we finally proved Theorem \ref{thm sec3 A0}.

\subsection{Velocity field}
Let $w(x,t)$ be a solution to the two-dimensional helical Euler equation \eqref{eq 2euler}. We define $$p(t):=\int_{\R^2}xw(x)\,dx,$$
we aim to estimate the derivative of $p(t)$, which has the form
\begin{equation*}
\frac{dp}{dt}=\int_{\R^2}w(x,t)Hw(x,t)\,dx.
\end{equation*}
Our main theorem in this subsection is
\begin{theorem}\label{thm sec3 A1}
For any $w\in L^1_4 \bigcap L^{\infty}_4$, there holds
\begin{equation*}\begin{aligned}
\int_{\R^2} w(x)Hw(x)\,dx=&-\frac{1}{4\pi}\int_{\R^2}\int_{\R^2}\log \frac{1}{|x-y|}\mathbf{1}_{|x-y|\le 1}\frac{x^{\perp}}{\langle x \rangle}  w(x)w(y)\,dx\,dy\\&+ O\left(\| w\|_{L^1_4(\R^2)}^2\right).
\end{aligned}\end{equation*}
\end{theorem}

Utilizing Proposition \ref{prop 3 euler}, first we need to estimate
$$
I_1(x,y):= \int_{-\pi}^{\pi} \frac{(x_1,x_2,0)-\left(R_{a}(y),a\right)}{\left( |x-R_ay|^2+a^2 \right)^\frac{3}{2}} \wedge \xi(\left(R_{a}(y),a\right))\,da.
$$
\begin{lemma} For any $x,y \in \R^2$ with $x\neq y$, there holds
\begin{equation*}
I_1(x,y)= \int_{-\pi}^{\pi} \frac{A^*_0(x,y)+a^2A^*_2(x,y)}{\left( |x-R_ay|^2+a^2 \right)^\frac{3}{2}} \,da+O\left( 1+|x|^3+|y|^3 \right),
\end{equation*}
where
$$
A^*_0(x,y):=\left( -(x-y)^{\perp},-y\cdot(x-y) \right),
$$
and
$$
A^*_2(x,y)=\left( \frac{x^{\perp}}{2},\frac{|x|^2}{2} \right) .
$$
\end{lemma}
\begin{proof}
Note that
$$
\int_{-\pi}^{\pi} \frac{|a|^3}{\left( |x-R_ay|^2+a^2 \right)^\frac{3}{2}} \,da \lesssim 1.
$$
By Taylor expansions, we obtain
$$
R_ay=e^{a\mathfrak{R}}x=y+a\mathfrak{R}x-\frac{a^2}{2}y+O(|a|^3|y|).
$$
Direct computation gives
\begin{equation*}
I_1(x,y)= \int_{-\pi}^{\pi} \frac{A_0(x,y)+aA_1(x,y)+a^2A_2(x,y)}{\left( |x-R_ay|^2+a^2 \right)^\frac{3}{2}} \,da+O\left( 1+|x|^2+|y|^2 \right),
\end{equation*}
where the coefficients are given by
$$
A_0(x,y):=\left( -(x-y)^{\perp},-y\cdot(x-y) \right),
$$
$$
A_1(x,y)=\left(0,0,y^{\perp}\cdot (x-y)\right)
$$
and
$$
A_2(x,y)=\left( \frac{y^{\perp}}{2},\frac{|y|^2}{2} \right) +\left(0,0,\frac{y\cdot (x-y)}{2}\right).
$$
By comparing $A_i$ with $A_i^*$, it remains to show that
\begin{equation}\label{eq sec42 A0}
\int_{-\pi}^{\pi} \frac{|a|\,|y|\,|x-y|+a^2\,|y|\,|x-y|+a^2\,|x-y|}{\left( |x-R_ay|^2+a^2 \right)^\frac{3}{2}} \,da \lesssim 1+|x|^2+|y|^2.
\end{equation}
Observing that for $|a| \le \pi$,
$$
\frac{|a|\,|y|\,|x-y|+a^2\,|y|\,|x-y|+a^2\,|x-y|}{\left( |x-R_ay|^2+a^2 \right)^\frac{3}{2}} \lesssim \frac{(1+|y|)\,|x-y|}{\left( |x-R_ay|^2+a^2 \right)},
$$
which gives \eqref{eq sec42 A0} by using \eqref{eq sec3 E0}.
\end{proof}
Recalling that $Hw(x)=(U^1,U^2)+x^{\perp}U^3$, the lemma above then implies that
\begin{equation}\begin{aligned}\label{eq sec3 E1}
\int_{\R^2} w(x)Hw(x)\,dx =& \frac{1}{4\pi}\int_{\R^2}\int_{\R^2}\int_{-\pi}^{\pi} \frac{(x-y)^{\perp}+\left( y\cdot (x-y) \right)x^{\perp}}{\left( |x-R_ay|^2+a^2 \right)^\frac{3}{2}} \,da\,w(x)w(y)\,dx \,dy\\
&-\frac{1}{8\pi}\int_{\R^2}\int_{\R^2}\int_{-\pi}^{\pi} \frac{a^2}{\left( |x-R_ay|^2+a^2 \right)^\frac{3}{2}} \,da\,\langle x \rangle^2 w(x)w(y)x^{\perp}\,dx \,dy \\
&+O\left( \|w\|_{L^1_4 \bigcap L^{\infty}_4} \right) \\
:=&\frac{1}{4\pi}\int_{\R^2}\int_{\R^2}\varphi(x,y)K_0^*(x,y)\,w(x)w(y)\,dx \,dy \\
&-\frac{1}{8\pi}\int_{\R^2}\int_{\R^2}K_2(x,y)\,\langle x \rangle^2 w(x)w(y)x^{\perp}\,dx \,dy \\
&+O\left( \|w\|_{L^1_4 \bigcap L^{\infty}_4} \right),
\end{aligned}\end{equation}
where
$$
\varphi(x,y):=(x-y)^{\perp}+\left( y\cdot (x-y) \right)x^{\perp},
$$
$$
K_0^*(x,y):=\int_{-\pi}^{\pi} \frac{1}{\left( |x-R_ay|^2+a^2 \right)^\frac{3}{2}} \,da
$$
and
$$
K_2(x,y):=\int_{-\pi}^{\pi} \frac{a^2}{\left( |x-R_ay|^2+a^2 \right)^\frac{3}{2}} \,da.
$$
Observe that $K_0^*(x,y)=K_0^*(y,x)$, we can rewrite
\begin{equation}\label{eq sec3 E3}\begin{aligned}
&\int_{\R^2}\int_{\R^2}\varphi(x,y)K_0^*(x,y)\,w(x)w(y)\,dx \,dy\\
=&\frac{1}{2}\int_{\R^2}\int_{\R^2}\left(\varphi(x,y)+\varphi(y,x)\right)K_0^*(x,y)\,w(x)w(y)\,dx \,dy.
\end{aligned}\end{equation}
Direct calculation shows that
\begin{equation}\label{eq sec3 E4}
\varphi(x,y)+\varphi(y,x)=-|x-y|^2\,x^{\perp}+(x\cdot(x-y))\,(x-y)^{\perp}.
\end{equation}
To further simplify the expression of \eqref{eq sec3 E1}, we will need the following lemma.
\begin{lemma}\label{le sec3 A0}
For any $x,y \in \R^2$ with $x \neq y$, there holds
\begin{equation*}
\int_{0}^{\infty} \frac{1}{\left( |x-R_ay|^2+a^2 \right)^\frac{3}{2}} \,da \lesssim 1+\frac{1+\min \{ |x|, \, |y| \}}{|x-y|^2}.
\end{equation*}
\end{lemma}
Combined this lemma with \eqref{eq sec3 E1}, \eqref{eq sec3 E3} and \eqref{eq sec3 E4}, we have
\begin{equation}\begin{aligned}\label{eq sec3 E2}
\int w(x)Hw(x)\,dx =&-\frac{1}{8\pi}\int_{\R^2}\int_{\R^2}K_2(x,y)\,\langle x \rangle^2 w(x)w(y)x^{\perp}\,dx \,dy \\
&+O\left( \|w\|_{L^1_4 \bigcap L^{\infty}_4} \right).
\end{aligned}\end{equation}

Before proving the lemma, we see that Theorem \ref{thm sec3 A1} follows directly from \eqref{eq sec3 E2} and Lemma \ref{le sec3 A0} once we establish the following:
\begin{lemma}\label{le sec3 A2}There exists a constant $\delta_1$  independent of $x,y$ such that
\begin{equation}\label{eq sec3.2 D0}
K_2(x,y)=\frac{2}{\langle x \rangle^3}\log\frac{1}{|x-y|}+O(\la y\ra)
\end{equation}
when $|x-y|\le \delta_1$ and
\begin{equation}\label{eq sec3.2 D1}
K_2(x,y) \lesssim \la y \ra
\end{equation}
when $|x-y|\ge \delta_1$.
\end{lemma}

Therefore, it remains to prove Lemma \ref{le sec3 A0} and \ref{le sec3 A2}.

\begin{proof}[Proof of Lemma \ref{le sec3 A0}]
We assume without loss of generality that $|y|\ge |x|$ (the case then $|y| <|x|$ can be treated similarly). First, we consider the case when $|y|\ge 2|x|$. In this case
 \[
\frac{|x-y|}{3} \le |x-R_ay| \le 3|x-y|,
\]
so we have
\[
\int_{0}^{\infty} \frac{1}{\left( |x-R_ay|^2+a^2 \right)^\frac{3}{2}}\,da \lesssim  \int_{0}^{\infty} \frac{1}{\left( |x-y|^2+a^2 \right)^\frac{3}{2}} \,da \approx \frac{1}{|x-y|^2} .
\]

Next, we consider the case when $2|x| \ge |y| \ge |x|$. On one hand, it follows from the Triangle inequality that
\begin{equation*}
|x-y| \le  \left|x-\frac{y}{|y|}|x|\right|+|y|-|x|.
\end{equation*}
On the other hand, the fact $|y| \ge |x|$ implies
\begin{equation*}
\left|x-\frac{y}{|y|}|x|\right| \le |x-y|
\end{equation*}
and
\begin{equation*}
|y|-|x|\le |x-y|,
\end{equation*}
which yield
\begin{equation*}
\left|x-\frac{y}{|y|}|x|\right|+|y|-|x| \lesssim |x-y|.
\end{equation*}
Gathering the estimates above, we finally obtain
\begin{equation}\label{eq distance funciton}
|x-y| \approx  \left|x-\frac{y}{|y|}|x|\right|+|y|-|x|.
\end{equation}
Next we set $\theta_{x,y}=\angle xoy$ and assume without loss of generality that $\theta_{x,y} \in [0,\pi]$, then
\begin{equation}\label{eq sec3.1 B1}
\theta_{x,y} \approx \frac{\left|x-\frac{y}{|y|}|x|\right|}{|x|}.
\end{equation}
\emph{Case 1: $\theta_{x,y}\ge \theta_0:=10^{-4}$.} In view of \eqref{eq distance funciton} and \eqref{eq sec3.1 B1}, there holds
\begin{equation*}
\left| x-R_ay \right| \approx |x-y|
\end{equation*}
when $a \le \frac{\theta_0}{2}$. Therefore,
\begin{equation*}\begin{aligned}
\int_0^{\infty} \frac{1}{\left( |x-R_ay|^2+a^2 \right)^\frac{3}{2}}\,da &\lesssim \int_0^{\frac{\theta_0}{2}}\frac{1}{\left( |x-y|^2+a^2 \right)^\frac{3}{2}}\,da +\int_{\frac{\theta_0}{2}}^\infty \frac{1}{a^3} \,da \\
&\lesssim \frac{1}{|x-y|^2}+1.
\end{aligned}\end{equation*}
\emph{Case 2: $\theta_{x,y}\le \theta_0$,}  the integral under consideration can be estimated as follows:
\begin{align}
\int_0^{\infty} \frac{1}{\left( |x-R_ay|^2+a^2 \right)^\frac{3}{2}}\,da \lesssim&
 \int_{\frac{\pi}{3}}^{\infty} \frac{1}{\left( |x-R_ay|^2+a^2 \right)^\frac{3}{2}}\,da\label{eq sec3.1 B2} \\ &+\int_{2\theta_{x,y}}^{\frac{\pi}{3}} \frac{1}{\left( |x-R_ay|^2+a^2 \right)^\frac{3}{2}}\,da\label{eq sec3.1 B3} \\ &+  \int_0^{2\theta_{x,y}} \frac{1}{\left( |x-R_ay|^2+a^2 \right)^\frac{3}{2}}\,da \label{eq sec3.1 B4}.
\end{align}
For \eqref{eq sec3.1 B2}, a direct calculation shows that
\begin{equation*}
\int_{\frac{\pi}{3}}^{\infty} \frac{1}{\left( |x-R_ay|^2+a^2 \right)^\frac{3}{2}}\,da \le \int_{\frac{\pi}{3}}^{\infty} \frac{1}{a^3}\,da \lesssim 1.
\end{equation*}
For \eqref{eq sec3.1 B3}, since $\frac{\pi}{3} \ge a \ge 2\theta_{x,y}$ implies $|x-R_ay| \ge |x-y|$, we have
\begin{equation*}
\int_{2\theta_{x,y}}^{\frac{\pi}{3}} \frac{1}{\left( |x-R_ay|^2+a^2 \right)^\frac{3}{2}}\,da \lesssim \int_0^{\infty} \frac{1}{\left( |x-y|^2+a^2 \right)^\frac{3}{2}} \,da \approx \frac{1}{|x-y|^2}.
\end{equation*}

For \eqref{eq sec3.1 B4}, due to the fact $|x-R_ay|=|x-R_{(2\theta_{x,y}-a)}y|$, we observe that
\begin{equation*}\begin{aligned}
\int_0^{2\theta_{x,y}} \frac{1}{\left( |x-R_ay|^2+a^2 \right)^\frac{3}{2}}\,da \le& \int_0^{\theta_{x,y}} \frac{1}{\left( |x-R_ay|^2+a^2 \right)^\frac{3}{2}}\,da \\ &+ \int_{\theta_{x,y}}^{2\theta_{x,y}} \frac{1}{\left( |x-R_ay|^2+a^2 \right)^\frac{3}{2}}\,da \\
  \le& 2\int^{\theta_{x,y}}_{0} \frac{1}{\left( |x-R_ay|^2+a^2 \right)^\frac{3}{2}}\,da.
\end{aligned}\end{equation*}

To estimate the right hand side, first we use \eqref{eq distance funciton} and \eqref{eq sec3.1 B1} to conclude that $$|y-x|\approx |x-\frac{y}{|y|}|x||+|y|-|x| \approx |x||\theta_{x,y}|+|y|-|x|,$$ which implies
\begin{equation*}\begin{aligned}
|x-R_ay| &\approx |x||\theta_{x,R_ay}|+|R_ay|-|x| \\  &\approx(\theta_{x,y}-a)|x|+|y|-|x|
\end{aligned}\end{equation*}
for $0 \le a \le \theta_{x,y}$. Thus,
\begin{equation*}\begin{aligned}
\int_0^{2\theta_{x,y}} \frac{1}{\left( |x-R_ay|^2+a^2 \right)^\frac{3}{2}}\,da  &\lesssim \int_0^{\theta_{x,y}} \frac{1}{\left((\theta_{x,y}-a)^2|x|^2+(|y|-|x|)^2+a^2\right)^{\frac32}}\,da \\
&= \theta_{x,y} \int_0^1 \frac{1}{\left((1-a)^2|\theta_{x,y}x|^2+(|y|-|x|)^2+|\theta_{x,y}a|^2\right)^{\frac32}}\,da.
\end{aligned}\end{equation*}
\emph{Case 2.1, $|y|-|x|\le \theta_{x,y}|x|$.} From \eqref{eq distance funciton} and \eqref{eq sec3.1 B1}, we see that $|y-x| \approx \theta_{x,y}|x|$ and consequently,
\begin{equation*}\begin{aligned}
\int_0^{2\theta_{x,y}} \frac{1}{\left( |x-R_ay|^2+a^2 \right)^\frac{3}{2}}\,da &\lesssim \theta_{x,y} \int_0^1 \frac{1}{\left((1-a)^2|\theta_{x,y}x|^2+|\theta_{x,y}a|^2\right)^{\frac32}}\,da \\
&= \frac{1}{\theta_{x,y}^2}\int_0^1 \frac{1}{\left((1-a)^2|x|^2+a^2\right)^{\frac32}}\,da.
\end{aligned}\end{equation*}
Then a direct calculation gives
\begin{equation*}\begin{aligned}
\int_0^1 \frac{1}{\left((1-a)^2|x|^2+a^2\right)^{\frac32}}\,da \approx&\int_0^{\frac12} \frac{1}{\left(|x|^2+a^2\right)^{\frac32}}\,da +\int_{\frac12}^1 \frac{1}{\left((1-a)^2|x|^2+1\right)^{\frac32}}\,da\\
=& \int_0^{\frac12} \frac{1}{\left(|x|^2+a^2\right)^{\frac32}}\,da +\int_0^{\frac12} \frac{1}{\left(a^2|x|^2+1\right)^{\frac32}}\,da \\
\approx&  \frac{1+|x|}{|x|^2},
\end{aligned}\end{equation*}
which implies (recalling that $|y-x| \approx \theta_{x,y}|x|$ and $|y| \approx |x|$ )
\begin{equation*}\begin{aligned}
\int_0^{2\theta_{x,y}} \frac{1}{\left( |x-R_ay|^2+a^2 \right)^\frac{3}{2}}\,da &\lesssim \frac{1+|x|}{\theta^2_{x,y}|x|^2} \lesssim \frac{\la y \ra}{|x-y|^2}.
\end{aligned}\end{equation*}

\emph{Case 2.2, $|y|-|x|\ge \theta_{x,y}|x|$.} It follows from \eqref{eq distance funciton} and \eqref{eq sec3.1 B1} that $|y-x| \approx |y|-|x|$. Therefore,
\begin{equation*}\begin{aligned}
\int_0^{2\theta_{x,y}} \frac{1}{\left( |x-R_ay|^2+a^2 \right)^\frac{3}{2}}\,da &\lesssim \theta_{x,y} \int_0^1 \frac{1}{\left((|y|-|x|)^2+|\theta_{x,y}a|^2\right)^{\frac32}}\,da \\
&\lesssim \int_0^{\infty} \frac{1}{\left((|y|-|x|)^2+a^2\right)^{\frac32}}\,da \\
&\lesssim \frac{1}{(|y|-|x|)2} \lesssim \frac{1}{|y-x|^2}.
\end{aligned}\end{equation*}
Gathering these estimates, we obtain the desired conclusion.

\end{proof}
\begin{proof}[Proof of Lemma \ref{le sec3 A2}] We adopt a similar argument to that used in the proof of Lemma \ref{le sec3 A1}. Let $\delta_0, \delta_1$ be the small constants chosen in Lemma \ref{le sec3 A1}, then \eqref{eq sec3 D1} follows directly from Lemma \ref{le sec3 A0}. Assuming that $|x-y|\le \delta_1$, $x^{\perp} \cdot y \ge 0$ and define
$$
\mathfrak{g}(x,y):= \sqrt{(1+x\cdot y)^2+(x^{\perp}\cdot y)^2},
$$
$$
\phi(x,y)=\arctan \frac{1+x\cdot y}{x^{\perp}\cdot y}.
$$
and
$$
\frac{\mathfrak{g}_1^2(x,y)}{2}:=\frac{4|x-y|^2+(|x|^2-|y|^2)^2}{2\mathfrak{g}(x,y)(|x|^2+|y|^2+2+2\mathfrak{g}(x,y))} \ge0.
$$
First we note that
$$
\left|\frac{\pi}{2}-\phi(x,y)\right|\lesssim \frac{x^{\perp}\cdot y}{1+x\cdot y},
$$
together with the fact
$$
\frac{1}{1+x\cdot y}=\frac{1}{1+|x|^2}+O\left( \frac{|x-y|}{1+|x|} \right),
$$
we see that
\begin{equation}\label{eq sec3.2 D2}
\left|\frac{\pi}{2}-\phi(x,y)\right|\lesssim |x-y|.
\end{equation}
Therefore, we obtain the following expression for $a^2$,
\begin{equation*}
a^2=\sin^2(a+\phi(x,y)-\frac{\pi}{2})+O(a|x-y|)+O\left(|x-y|^2\right).
\end{equation*}
Utilizing Lemma \ref{le sec3 A1}, we arrive at
\begin{equation*}\begin{aligned}
\int_{-\pi}^{\pi} \frac{|a|\,|x-y|}{\left( |x-R_ay|^2+a^2 \right)^\frac{3}{2}}\,da \lesssim& \int_{-\pi}^{\pi} \frac{|x-y|}{ |x-R_ay|^2+a^2 }\,da \lesssim 1.
\end{aligned}\end{equation*}
Moreover, Lemma \ref{le sec3 A0} gives
\begin{equation*}\begin{aligned}
\int_{-\pi}^{\pi} \frac{|x-y|^2}{\left( |x-R_ay|^2+a^2 \right)^\frac{3}{2}}\,da \lesssim \la y \ra,
\end{aligned}\end{equation*}
which implies
\begin{equation*}\begin{aligned}
K_2(x,y)=\int_{-\pi}^{\pi} \frac{a^2}{\left( |x-R_ay|^2+a^2 \right)^\frac{3}{2}}\,da =& \int_{-\pi}^{\pi} \frac{\sin^2(a+\phi(x,y)-\frac{\pi}{2})}{\left( |x-R_ay|^2+a^2 \right)^\frac{3}{2}}\,da \\&+O(\la y \ra).
\end{aligned}\end{equation*}
Arguing as in the proof of Lemma \ref{le sec3 A1}, we obtain
\begin{equation*}\begin{aligned}
K_2(x,y)=&\int_{-\delta_0}^{\delta_0} \frac{\sin^2(a+\phi(x,y)-\frac{\pi}{2})}{\left( |x-R_{a}y|^2+2-2\cos a \right)^{3/2}}\,da \\&+\int_{-\delta_0}^{\delta_0} \frac{O\left(a^2\sin^2(a+\phi(x,y)-\frac{\pi}{2})\right)}{\left( |x-R_{a}y|^2+2-2\cos a \right)^{3/2}} \,da+O(\la y \ra).
\end{aligned}\end{equation*}
Observe that (using \eqref{eq sec3.2 D2})
$$
\frac{a^2\sin^2(a+\phi(x,y)-\frac{\pi}{2})}{\left( |x-R_{a}y|^2+2-2\cos a \right)^{3/2}} \lesssim 1+\frac{|x-y|}{\left( |x-R_{a}y|^2+2-2\cos a \right)^{1/2}},
$$
which, combined with \eqref{eq sec3 E0}, gives
\begin{equation*}\begin{aligned}
K_2(x,y)=\int_{-\delta_0}^{\delta_0} \frac{\sin^2(a+\phi(x,y)-\frac{\pi}{2})}{\left( |x-R_{a}y|^2+2-2\cos a \right)^{3/2}}\,da +O(\la y \ra).
\end{aligned}\end{equation*}
Again, arguing as in the proof of Lemma \ref{le sec3 A1}, we see
\begin{equation*}\begin{aligned}
K_2(x,y)&=\int_{-\pi}^{\pi} \frac{\sin^2(a+\phi(x,y)-\frac{\pi}{2})}{\left( |x|^2+|y|^2+2-2\mathfrak{g}(x,y) \cos\left(a+\phi(x,y)-\frac{\pi}{2}\right) \right)^{3/2}}\,da +O(\la y \ra) \\
&=\frac{1}{\left(2\mathfrak{g}(x,y)\right)^{\frac32}}\int_{-\frac{3\pi}{2}+\phi(x,y)}^{\frac{\pi}{2}+\phi(x,y)} \frac{\sin^2 a}{\left(\frac{\mathfrak{g}_1^2(x,y)}{2}+1-\cos a\right)^{\frac32}}\,da+O(\la y \ra) \\
&=\frac{1}{\left(2\mathfrak{g}(x,y)\right)^{\frac32}}\int_{-\pi}^{\pi} \frac{a^2+O\left( |a|^3 \right)}{\left(\frac{\mathfrak{g}_1^2(x,y)}{2}+\frac{a^2}{2}\right)^{\frac32}}\,da+O(\la y \ra) \\
&=\frac{1}{\left(2\mathfrak{g}(x,y)\right)^{\frac32}}\int_{-\pi}^{\pi} \frac{a^2}{\left(\frac{\mathfrak{g}_1^2(x,y)}{2}+\frac{a^2}{2}\right)^{\frac32}}\,da+O(\la y \ra).
\end{aligned}\end{equation*}
Following the same calculation as that in Lemma \ref{le sec3 A1}, we get the desired conclusion.

\end{proof}

\section{Physical quantities for helical Euler equation}
In this section, we will introduce the energy, momentum and the distance function associated with the helical Euler equation. Let $w_{\e}(x,t)$ solve the helical Euler equation \eqref{eq 2euler} with initial data $w_{\e,0}$. We define the energy
\begin{equation*}
E(t):=\int_{\R^2\times \T} (-\Delta_{\R^2\times \T})^{-1}\Omega \cdot \Omega \,dx,
\end{equation*}
the modified energy
\begin{equation*}
E^*(t):=\int_{\R^2}\int_{\R^2}\log \frac{1}{|x-y|}\mathbf{1}_{|x-y|\le 1}\langle x \rangle  w(x)w(y)\,dx\,dy,
\end{equation*}
the second momentum
\begin{equation*}
M_1(t):=\int_{\R^2} |x|^2w_{\e}(x,t)\,dx,
\end{equation*}
the high order momentum
\begin{equation*}
M_2(t):=\int_{\R^2} |x|^4w_{\e}(x,t)\,dx,
\end{equation*}
the center of gravity (assuming that $\int w_{\e}(x,t)\,dx=1$)
\begin{equation*}
p(t):=\int_{\R^2} xw_{\e}(x,t)\,dx,
\end{equation*}
the distance function
\begin{equation*}
D_1(t):=\int_{\R^2} |x-p(t)|^2w_{\e}(x,t)\,dx,
\end{equation*}
and
\begin{equation*}
D_2(t):=\int_{\R^2}\int_{\R^2} |x-y|^2w_{\e}(x,t)w_{\e}(y,t)\,dx\,dy.
\end{equation*}
We also define the following quantity which will be used to control the momentum outside origin
\begin{equation*}
M^*(t):=\int_{\R^2} |x|^2 \eta(|x|)w_{\e}(x,t)\,dx,
\end{equation*}
where $0\le \eta \le 1$ is a smooth non-negative cut-off function satisfies
\begin{equation}\label{eq def of eta}\eta(r)=\begin{cases}
0 \quad r \le 10 \\
1 \quad r>20.
\end{cases}\end{equation}
For $t=0$, we assume that the initial vorticity $w_{\e,0}(x)$  satisfies following conditions, (these conditions are natural and can be verified directly if the initial data has the form $w_{\e,0}(x)=\frac{1}{\e^2}\eta(\frac{x-(1,0)}{\e})$).
\begin{assumption}\label{as initial data}
We assume that $0 \le w_{\e,0}(x) \le \frac{C}{\e^2}$, $\int w_{\e,0}(x)\,dx=1$ and

\begin{equation*}\begin{aligned}
E(0)&=\sqrt2 \log \frac{1}{\e}+O(1) \\
M_1(0)&=1+O(\e) \\
M_2(0)&=1+O(\e) \\
p(0)&=(1,0)
\end{aligned}\end{equation*}
and
\begin{equation}\label{eq sec4.1 C9}
M^*(0) =O(\e).\,\,\,\,\,\,\,\,\,\,\,\,\,\,\,\,\,\,\,\,\,\,\,\,\,\,
\end{equation}
\end{assumption}
Then we will consider these physical quantities for $t\in [0,T)$. It is worth noting that the time $T>0$ is a fixed constant, so we have $e^{Ct} =1+O(t)$, which will be frequently used in our subsequent discussions. Now we present the main theorem of this section
\begin{theorem}
Let $w_{\e}(x,t)$ solve the helical Euler equation \eqref{eq 2euler} with initial data $w_{\e,0}$. Assuming that $w_{\e,0}$ satisfies Assumption \ref{as initial data}, then there hold
\begin{equation}\label{eq sec4 C5}
0 \le w_{\e}(x,t) \le \frac{C}{\e^2},\,\,\,\, \int_{\R^2} w_{\e}(x,t)\,dx=1
\end{equation}
and
\begin{align}
E(t)&=\sqrt2 \log \frac{1}{\e}+O(1) \label{eq sec4 C0} \\
E^*(t)&=\sqrt2 \log \frac{1}{\e}+O(1) \label{eq sec4 C1} \\
M_1(t)&=1+O(\e) \label{eq sec4 C2} \\
M_2(t)&=1+O(t). \label{eq sec4 C3}
\end{align}
\end{theorem}
\begin{proof}
Recall that $\Omega_{\e}^z(x',x_3,t):=w_{\e}(R_{-x_3}x',t)$ satisfies the three-dimensional helical Euler equation \eqref{eq 3euler z} and it follow from Lemma \ref{le sec2 C7} that
$$
\int_{\R^2}\phi(|x|)w_{\e}(x,t)\,dx=\frac{1}{2\pi}\int_{\R^2\times \T} \phi(|x'|)\Omega^z_{\e}(x,t)\,dx
$$
holds for any smooth function $\phi$. Therefore, \eqref{eq sec4 C0} and \eqref{eq sec4 C2} follow directly from Theorem \ref{thm main} and Assumption \ref{as initial data}. Moreover, \eqref{eq sec4 C5} holds trivially since $\Omega_{\e}^z(x,t)$ satisfies a transport equation and the velocity field is divergence free, as stated in \cite{GZ1}. By utilizing \eqref{eq sec4 C0}, \eqref{eq sec4 C2} and Theorem \ref{thm sec3 A0}, we obtain \eqref{eq sec4 C1}. Now it remains to prove \eqref{eq sec4 C3}. First, for any smooth function $\phi$, there hold
\begin{equation*}\begin{aligned}
\frac{d}{dt}\int_{\R^2\times \T} \phi(|x'|)\Omega^z(x,t)\,dx &=-\int_{\R^2\times \T} \phi(|x'|)\nabla \cdot \left( U(x,t)\Omega^z(x,t) \right)\,dx \\
&=\int_{\R^2 \times \T}\dot{\phi}(|x'|)\frac{(x',0)}{|x'|}\cdot U(x,t) \Omega^z (x,t)\,dx.
\end{aligned}\end{equation*}
Recall that
$$
U(x,t)=\int_{\R^2 \times \T} \nabla_xG(x-y) \wedge \xi(y) \Omega^z(y,t)\,dy,
$$
which allows us to rewrite the integral in the following form:
\begin{equation*}\begin{aligned}
\frac{d}{dt}\int_{\R^2\times \T} \phi(|x'|)\Omega^z(x,t)\,dx & =\int_{\R^2 \times \T}\int_{\R^2 \times \T}\frac{\dot{\phi}(|x'|)}{|x'|} \mathfrak{g}(x,y) \Omega^z(y,t) \Omega^z(x,t)\,dy\,dx,
\end{aligned}\end{equation*}
where
\begin{equation}\label{eq sec4 C8}\begin{aligned}
\mathfrak{g}(x,y):=&(x',0) \cdot \nabla G(x-y) \wedge \xi(y)\\
=&\Big(x_1\partial_2G(x-y)-x_2\partial_1G(x-y)\Big) +(x_1y_1+x_2y_2)\partial_3G(x-y).
\end{aligned}\end{equation}
It has been shown in \cite{GZ1} that
$$
\mathfrak{g}(x,y)+\mathfrak{g}(y,x) \equiv 0,
$$
which yields
\begin{equation}\label{eq sec4 C9}\begin{aligned}
&\frac{d}{dt}\int_{\R^2\times \T} \phi(|x'|)\Omega^z(x,t)\,dx   \\
=&\frac12 \int_{\R^2 \times \T}\int_{\R^2 \times \T}\left(\frac{\dot{\phi}(|x'|)}{|x'|}-\frac{\dot{\phi}(|y'|)}{|y'|}\right) \mathfrak{g}(x,y) \Omega^z(y,t) \Omega^z(x,t)\,dy\,dx.
\end{aligned}\end{equation}
Indeed, conservation of the second momentum is obtained by setting $\phi(r)=r^2$. Now, if we set $\phi(r)=r^4$, \eqref{eq sec4 C9} implies
\begin{equation}\label{eq sec4 C7}\begin{aligned}
\left|\frac{dM_2(t)}{dt}\right| \lesssim& \int_{\R^2 \times \T}\int_{\R^2 \times \T}|x-y|(|x|+|y|) |\mathfrak{g}(x,y)| \Omega_{\e}^z(y,t) \Omega_{\e}^z(x,t)\,dy\,dx.
\end{aligned}\end{equation}
Using Lemma \ref{le sec3 G}, we observe that
\begin{equation}\label{eq sec4.3 A0}
|\nabla G(x-y)| \lesssim 1+\frac{1}{|x-y|}.
\end{equation}
Combined with \eqref{eq sec4 C8}, there holds
$$
|x-y| \,(|x|+|y|)\, |\mathfrak{g}(x,y)| \lesssim 1+|x|^4+|y|^4.
$$
Therefore, we use \eqref{eq sec4 C5} and \eqref{eq sec4 C7} to conclude that
\begin{equation*}\begin{aligned}
\left|\frac{dM_2(t)}{dt}\right| \lesssim& 1+M_2(t),
\end{aligned}\end{equation*}
which gives \eqref{eq sec4 C3} by using Gronwall's inequality.
\end{proof}
To control the momentum outside origin, we require the following useful lemma.
\begin{lemma}\label{le sec4.3 C8} For any $t \in [0,T)$, there holds
\begin{equation}\label{eq sec4.3 B3}
\frac{dM^*}{dt} \lesssim M^*(t)+\int_{|x|\ge 10}\la x\ra w(x,t)\,dx.
\end{equation}
\end{lemma}
\begin{proof}
By setting $\phi(r)=r^2\eta(r)$ in \eqref{eq sec4 C9}, we obtain
\begin{equation*}
\frac{dM^*}{dt}=\frac{1}{8\pi^2} \int_{\R^2 \times \T}\int_{\R^2 \times \T}(F(|x'|)-F(|y'|)) \mathfrak{g}(x,y) \Omega^z(y,t) \Omega^z(x,t)\,dy\,dx
\end{equation*}
where
$$
F(r):= 2\eta(r)+r\dot{\eta}(r).
$$

Using \eqref{eq sec4 C5}, \eqref{eq sec4 C8} and \eqref{eq sec4.3 A0}, it follows that

\begin{equation*}\begin{aligned}
\frac{dM^*}{dt}\lesssim& \int_{\R^2 \times \T}\int_{\R^2 \times \T} \bigg| F(|x'|)-F(|y'|) \bigg| \left(  1+\frac{1}{|x-y|} \right)|x'|(1+|y'|) \Omega^z(x,t) \Omega^z(y,t)\,dx\,dy \\
:=& \int_{\left(\R^2\times \T\right)^2}  I(x,y,t) \,dx \,dy.
\end{aligned}\end{equation*}
To proceed, we first consider the integral in the region $|x-y|\le1$. In this region, $F(|x'|)-F(|y'|) \equiv 0$ if $|x'|\le 10$ and $|y'|\le 10$, which implies
\begin{equation*}\begin{aligned}
 \int_{|x-y|\le1} I(x,y,t)\,dx\,dy \le&  \int_{|x-y|\le1 , \, |x'|\le 10\, , |y|' \ge 10} I(x,y,t)\,dx\,dy\\
&+\int_{|x-y|\le1, \, |x'| \ge 10} I(x,y,t)\,dx\,dy.
\end{aligned}\end{equation*}
Noting that $\dot{F}$ is bounded, we obtain the following estimate:
$$
\bigg| F(|x'|)-F(|y'|) \bigg| \left(  1+\frac{1}{|x-y|} \right) \lesssim \frac{\bigg| F(|x'|)-F(|y'|) \bigg|}{\big||x'|-|y'|\big|} \lesssim 1.
$$
Recall that $\Omega^z(x,t)=w(R_{-x_3}x')$, and utilizing \eqref{eq sec4 C5} and \eqref{eq sec4 C3}, we can further deduce
\begin{equation}\label{eq sec4.3 B0}\begin{aligned}
 \int_{|x-y|\le1 } I(x,y,t)\,dx\,dy  \lesssim&  \int_{|y'|\ge 10} (1+|y'|) \Omega^z(y,t)\,dy \int_{\R^2\times \T} |x'| \Omega^z(x,t)\,dx \\
+&\int_{\R^2\times \T} (1+|y'|) \Omega^z(y,t)\,dy \int_{|x'|\ge 10} |x'| \Omega^z(x,t)\,dx \\
\lesssim& \int_{|x|\ge 10}\la x\ra w(x,t)\,dx.
\end{aligned}\end{equation}
Next, considering the region where $|x-y|\ge 1$, we see that $1+\frac{1}{|x-y|} \lesssim 1$, which implies
\begin{equation}\label{eq sec4.3 B1}\begin{aligned}
\int_{|x-y|\ge 1}I(x,y,t) \,dx\,dy \lesssim& \int_{\R^2\times \T} \big| F(|x'|)\big|\, |x'|\Omega^z(x,t)\,dx \int_{\R^2\times \T} \la y \ra \Omega^z(y,t)\,dy \\
&+\int_{\R^2\times \T}  |x'|\Omega^z(x,t)\,dx \int_{\R^2\times \T} \big| F(|y'|)\big|\,\la y \ra \Omega^z(y,t)\,dy \\
\lesssim& \int_{\R^2} |F(|x|)|\la x\ra w(x,t)\,dx \\
\lesssim& \int_{\R^2} \eta(|x|)\la x \ra w(x,t)\,dx+\int_{\R^2} |\dot{\eta}(x)|\la x\ra^2 w(x,t)\,dx.
\end{aligned}\end{equation}
Since $\eta(|x|)$ is supported in the region $|x| \ge 10$ and $\dot{\eta}(|x|)$ is supported in the region $10 \le |x| \le 20$, we have
$$
\la x \ra \eta(|x|) \le |x|^2 \eta(|x|)
$$
and
$$
|\dot{\eta}(x)|\la x\ra^2 \lesssim \mathbf{1}_{|x|\ge 10}(x),
$$
which, combined with \eqref{eq sec4.3 B1}, yields
$$
\int_{|x-y|\ge 1}I(x,y,t) \,dx\,dy \lesssim M^*(t)+\int_{|x|\ge 10}\la x\ra w(x,t)\,dx
$$
and hence completes the proof.
\end{proof}

Finally, we establish several straightforward equalities concerning the second momentum, center of gravity and distance function. These equalities are crucial for us to understand the long-time dynamics for helical vortex filament.
\begin{theorem}
Under Assumption \ref{as initial data}, we have
\begin{equation}\label{eq sec4 D0}
D_1(t)+|p(t)|^2=M_1(t),
\end{equation}
and
\begin{equation}\label{eq sec4 D1}
D_2(t)+2|p(t)|^2=2M_1(t).
\end{equation}
In particular, there holds
\begin{equation}\label{eq sec4 D2}
|p(t)|\le 1+O(\e).
\end{equation}
\end{theorem}
\begin{proof}
Equation \eqref{eq sec4 D0} and \eqref{eq sec4 D1} can be checked directly. Equation \eqref{eq sec4 D2} follows directly from \eqref{eq sec4 C5}, \eqref{eq sec4 C2} and \eqref{eq sec4 D0}.
\end{proof}
\section{Long time dynamics for helical vortex filament}
Our main task of this section is to prove Theorem \ref{thm GZ main}. Let $w_{\e}(\cdot,t) \in L^1_4\bigcap L^{\infty}_4$ be a solution to the two-dimensional helical Euler equation \eqref{eq 2euler}. We assume throughout this section that its initial data $w_{\e,0}(x)$ satisfies Assumption \ref{as initial data}.

\subsection{Concentration of the vorticity}
Firstly, we demonstrate the existence of a point $p_{\e}^*(t) \in \R^2$ such that $w_{\e}(\cdot,t)$ concentrated near $p^*_{\e}(t)$. Inspired by the argument in \cite{Tur}, we define the set
$$
\Omega^{\e}_{\lambda}(t):=\left\{ x \bigg| Gw_{\e}(x,t)>\log \frac{1}{\e}-\frac{A}{\lambda_{\e}} \right\},
$$
where $A$ is a fixed constant that we will specified later. Here $Gw_{\e}(x,t)$ is defined by
$$
Gw_{\e}(x,t):=\int_{\R^2} \log \frac{1}{|x-y|}\mathbf{1}_{|x-y|\le 1} w_{\e}(y,t)\,dy.
$$
We aim to choose suitable $\lambda_{\e} \approx \frac{1}{\log \frac{1}{\e}}$ such that when $\e$ is small enough, there hold
\begin{equation}\label{eq sec5 A0}
\int_{\Omega^{\e}_{\lambda}(t)}w_{\e}(x,t)\,dx \ge 1-\lambda_{\e}
\end{equation}
and
\begin{equation}\label{eq sec5 A1}
diam \,\,\,\Omega^{\e}_{\lambda}(t) \le \e^{1-\alpha}
\end{equation}
for any $1\ge\alpha>0$. Once \eqref{eq sec5 A0} and \eqref{eq sec5 A1} are established, we can choose $p_{\e}^*(t) \in \Omega^{\e}_{\lambda}(t)$ and it holds that
\begin{equation}\label{eq sec5 A2}
\int_{|x-p_{\e}^*(t)|\le \e^{1-\alpha}} w_{\e}(x,t)\,dx \ge 1-\frac{C}{\log \frac{1}{\e}}.
\end{equation}
We begin with the estimates on $Gw_{\e}(x)$ and we will omit the time parameter to simplify the notation.
\begin{equation*}\begin{aligned}
Gw_{\e}(x)&=\int_{\R^2} \log \frac{1}{|x-y|}\mathbf{1}_{|x-y|\le 1} w_{\e}(y)\,dy \\
&=\int_{\R^2} \left( \log \frac{1}{\e}+\log \frac{1}{|\frac{x}{\e}-y|} \right)\mathbf{1}_{|x-\e y|\le 1} \tilde{w}_{\e}(y)\,dy,
\end{aligned}\end{equation*}
where
$$
\tilde{w}_{\e}(x):= \e^2w_{\e}(\e x).
$$
Utilizing the assumption \eqref{eq sec4 C5}, we obtain
\begin{equation}\label{eq sec5 A4}
\|\tilde{w}_{\e}\|_{L^1\bigcap L^{\infty}} \lesssim 1.
\end{equation}
Consequently,
$$
Gw_{\e}(x)\le \log \frac{1}{\e}+\int_{\R^2}\log \frac{1}{|\frac{x}{\e}-y|} \mathbf{1}_{|x-\e y|\le 1} \tilde{w}_{\e}(y)\,dy.
$$
Note that
$$
\log \frac{1}{|\frac{x}{\e}-y|} \mathbf{1}_{|x-\e y|\le 1} \,\,\le \frac{1}{|\frac{x}{\e}-y|},
$$
we have
\begin{equation}\label{eq sec5 A3}
Gw_{\e}(x)\le \log \frac{1}{\e}+C\|\tilde{w}_{\e}\|_{L^1\bigcap L^{\infty}}\le \log \frac{1}{\e}+C_1.
\end{equation}
Next we use \eqref{eq sec4 C1} to conclude that
\begin{equation*}
\int_{\R^2} \langle x\rangle w_{\e}(x)Gw_{\e}(x)\, dx \ge \sqrt2\log \frac{1}{\e}-C_2
\end{equation*}
for some $C_2>0$. Therefore,
\begin{equation*}\begin{aligned}
0&\le \int_{\R^2} \langle x\rangle w_{\e}(x)Gw_{\e}(x)\, dx - \sqrt2\log \frac{1}{\e}+C_2 \\
&=\int_{\R^2} \langle x\rangle w_{\e}(x)\left( Gw_{\e}(x)-\log \frac{1}{\e}\right) \, dx+\log \frac{1}{\e} \left( \int_{\R^2} \langle x\rangle w_{\e}(x)\, dx-\sqrt2 \right)+C_2.
\end{aligned}\end{equation*}
A key observation is that, by applying Jensen's inequality, \eqref{eq sec4 C5} and \eqref{eq sec4 C2}, there holds
\begin{equation*}\begin{aligned}
\int_{\R^2} \langle x\rangle w_{\e}(x)\, dx-\sqrt2 \le \left(\int_{\R^2}(1+|x|^2)w_{\e}(x)\,dx\right)^{\frac{1}{2}}-\sqrt2 \le C\e.
\end{aligned}\end{equation*}
Thus,
\begin{equation*}\begin{aligned}
0\le \int_{\R^2} \langle x\rangle w_{\e}(x)\left( Gw_{\e}(x)-\log \frac{1}{\e}\right) \, dx+C_3,
\end{aligned}\end{equation*}
which implies
\begin{equation*}\begin{aligned}
\int_{\R^2 \backslash \Omega^{\e}_{\lambda}} \langle x\rangle w_{\e}(x)\left(\log \frac{1}{\e}- Gw_{\e}(x)\right) \, dx\le \int_{\Omega^{\e}_{\lambda}} \langle x\rangle w_{\e}(x)\left( Gw_{\e}(x)-\log \frac{1}{\e}\right) \, dx +C_3.
\end{aligned}\end{equation*}
Using \eqref{eq sec5 A3} and the definition of $\Omega^{\e}_{\lambda}$, we arrive at
\begin{equation*}\begin{aligned}
\frac{A}{\lambda_{\e}}\int_{\R^2 \backslash \Omega^{\e}_{\lambda}} \langle x\rangle w_{\e}(x)\,dx \le C_3+C_1\|w_{\e}\|_{L^1_1} \le C_4.
\end{aligned}\end{equation*}
By setting $A=C_4$, it follows that
\begin{equation}\label{eq sec5 A6}
\int_{\R^2 \backslash \Omega^{\e}_{\lambda}} \langle x\rangle w_{\e}(x)\,dx \le \lambda_{\e},
\end{equation}
which proves \eqref{eq sec5 A0} by invoking \eqref{eq sec4 C5}.

\quad Now it remains to show that for any $0< \alpha \le1$, there holds
$$
diam \,\,\, \Omega^{\e}_{\lambda} \le 2\e^{1-\alpha}
$$
when $\e$ is small enough. By definition of $\Omega^{\e}_{\lambda}$, there holds
\begin{equation*}\begin{aligned}
0&< Gw_{\e}(x)-\log \frac{1}{\e}+\frac{A}{\lambda_{\e}} \\
&=\int_{\R^2} \left(\log \frac{1}{|x-y|}\mathbf{1}_{|x-y|\le 1}-\log \frac{1}{\e}\right) w_{\e}(y)\,dy+\frac{A}{\lambda_{\e}}
\end{aligned}\end{equation*}
whenever $x \in \Omega^{\e}_{\lambda}$. Thus, for any $\frac{1}{\e} \ge R \ge 1$, we have
\begin{equation}\label{eq sec5 A9}\begin{aligned}
&\int_{|y-x|\ge R\e} \left(\log \frac{1}{\e}-\log \frac{1}{|x-y|}\mathbf{1}_{|x-y|\le 1}\right) w_{\e}(y)\,dy \\ \le& \int_{|y-x|\le R\e} \left(\log \frac{1}{|x-y|}\mathbf{1}_{|x-y|\le 1}-\log \frac{1}{\e}\right) w_{\e}(y)\,dy+\frac{A}{\lambda_{\e}}.
\end{aligned}\end{equation}
On one hand,
\begin{equation}\label{eq sec5 A8}\begin{aligned}
&\int_{|y-x|\ge R\e} \left(\log \frac{1}{\e}-\log \frac{1}{|x-y|}\mathbf{1}_{|x-y|\le 1}\right) w_{\e}(y)\,dy \\=& \int_{R\e \le |x-y|\le 1} \left( \log \frac{1}{\e}-\log \frac{1}{|x-y|} \right)w_{\e}(y)\,dy \\
&+ \log \frac{1}{\e}\int_{|x-y|\ge 1}w_{\e}(y)\,dy \\
\ge& \log R \int_{R\e \le |x-y|\le 1}w_{\e}(y)\,dy +\log \frac{1}{\e}\int_{|x-y|\ge 1}w_{\e}(y)\,dy \\
\ge& \log R \int_{R\e \le |x-y|}w_{\e}(y)\,dy
\end{aligned}\end{equation}
since $\log R \le \log \frac{1}{\e}$. On the other hand,
\begin{equation*}\begin{aligned}
 &\int_{|y-x|\le R\e} \left(\log \frac{1}{|x-y|}\mathbf{1}_{|x-y|\le 1}-\log \frac{1}{\e}\right) w_{\e}(y)\,dy \\
 \le& \int_{|x-y|\le R\e}\left( \log \frac{1}{|x-y|}-\log \frac{1}{\e} \right)\mathbf{1}_{|x-y|\le 1}w_{\e}(y)\,dy \\
 =& \int_{|x-\e y|\le R\e}\log \frac{1}{|\frac{x}{\e}-y|}\mathbf{1}_{|x-\e y|\le 1}\tilde{w}_{\e}(y)\,dy \\
 \le& \int_{\R^2} \frac{\tilde{w}_{\e}(y)}{|\frac{x}{\e}-y|}\,dy \le \|\tilde{w}_{\e}\|_{L^1\bigcap L^{\infty}} \le C_6.
\end{aligned}\end{equation*}
Together with \eqref{eq sec5 A9} and \eqref{eq sec5 A8}, we see that
$$
\int_{|x-y|\ge R\e}w_{\e}(y)\,dy \le \frac{C_6+\frac{A}{\lambda_{\e}}}{\log R} \le \frac{C_7}{\lambda_{\e} \log R}
$$
provided that $\lambda_{\e} \le \frac{1}{2}$. By setting $\lambda_{\e}=\frac{4 C_7}{\alpha\log \frac{1}{\e}}$ and $R=e^{\frac{4C_7}{\lambda_{\e}}}=\frac{1}{\e^{\alpha}}$, (which satisfies $1\le R \le \frac{1}{\e}$) there holds
\begin{equation}\label{eq sec5 A7}
\int_{|x-y|\ge \e^{1-\alpha}}w_{\e}(y)\,dy \le \frac{1}{4},
\end{equation}
which implies
$$
\int_{|x-y|\le \e^{1-\alpha}}w_{\e}(y)\,dy \ge \frac{3}{4}
$$
for any $x \in \Omega^{\e}_{\lambda}$. Assuming that $diam \,\,\, \Omega^{\e}_{\lambda} \ge 2\e^{1-\alpha}$, then there exist $x_1, x_2 \in \Omega^{\e}_{\lambda}$ such that
$$B_{\e^{1-\alpha}(x_1)} \bigcap B_{\e^{1-\alpha}(x_2)}=\varnothing.$$
Thus,
$$
\int_{\R^2} w_{\e}(y)\,dy \ge \int_{B_{\e^{1-\alpha}}(x_1)}w_{\e}(y)\,dy +\int_{B_{\e^{1-\alpha}}(x_2)}w_{\e}(y)\,dy \ge \frac{3}{2}\,\,\,\,,
$$
which contradicts the assumption \eqref{eq sec4 C5} and hence completes the proof. Moreover, \eqref{eq sec5 A6} gives
\begin{equation}\label{eq sec5 B4}
\int_{|x-p_{\e}^*(t)|\ge \e^{1-\alpha}}\langle x\rangle w_{\e}(x,t)\,dx \le \frac{C_{8}}{\alpha\log \frac{1}{\e}}
\end{equation}
for any $0<\alpha\le 1 $ and $t >0$, provided that $\e$ is chosen small enough.

\subsection{Dynamics for $p_{\e}^*(t)$}
The main purpose of this subsection is to show that
\begin{theorem}For any $t \in [0,T)$ and $0<\alpha <1$, by choosing $\e$ small enough, there hold
\begin{equation}\label{eq sec5 B3}
|p_{\e}(t)-p_{\e}^*(t)| \le  \frac{C}{\log \frac{1}{\e}},
\end{equation}
\begin{equation}\label{eq sec5.2 C8}
\int_{|x-p_{\e}^*(t)|\ge \e^{1-\alpha}}\langle x\rangle^2 w_{\e}(x,t)\,dx \lesssim \frac{1}{\alpha\log \frac{1}{\e}},
\end{equation}
\begin{equation}\label{eq sec5 C0}
|p_{\e}(t)|=1+O\left(\frac{1}{\log\frac{1}{\e}}\right)
\end{equation}
and
\begin{equation}\label{eq sec5 C1}
D_2(t)=O\left(\frac{1}{\log\frac{1}{\e}}\right).
\end{equation}
\end{theorem}
\begin{proof}
First we show that $|p_{\e}^*(t)| \le 100$. Otherwise, we have
\begin{equation}\label{eq sec5 E0}
p_{\e}=\int_{|x-p^*_{\e}|\ge 1} xw_{\e}(x)\,dx+\int_{|x-p^*_{\e}|< 1} xw_{\e}(x)\,dx.
\end{equation}
For the first term on the right-hand side, Jensen's inequality yields
\begin{equation}\label{eq sec5 E1}
\left|\int_{|x-p^*_{\e}|\ge 1} xw_{\e}(x)\,dx\right| \le \left(\int_{\R^2} |x|^2 w_{\e}(x)\,dx\right)^{\frac12} \le2.
\end{equation}

For the second term, we have
\begin{equation}\label{eq sec5 A5}\begin{aligned}
\int_{|x-p^*_{\e}|< 1} xw_{\e}(x)\,dx&=\int_{|x-p^*_{\e}|< 1} (x-p^*_{\e})w_{\e}(x)\,dx+p^*_{\e}\int_{|x-p^*_{\e}|< 1} w_{\e}(x)\,dx.
\end{aligned}\end{equation}
By utilizing \eqref{eq sec4 C5}, it follows that
$$
\left|\int_{|x-p^*_{\e}|< 1} (x-p^*_{\e})w_{\e}(x)\,dx\right| \le \int_{\R^2} w_{\e}(x)\,dx \le2.
$$
Using \eqref{eq sec5 A2} and our assumption $|p^*_{\e}|\ge 100$, we obtain
$$
\left|p^*_{\e}\int_{|x-p^*_{\e}|< 1} w_{\e}(x)\,dx \right| \ge 50,
$$
which, together with \eqref{eq sec5 E0}-\eqref{eq sec5 A5}, implies
$$
|p_{\e}|\ge 40.
$$
This leads to a contradiction since \eqref{eq sec4 D2} states that
\begin{equation}\label{eq sec5 B6}
|p_{\e}|\le 1+O(\e).
\end{equation}
Therefore, we have shown that $|p^*_{\e}|\le 100$. Next we prove \eqref{eq sec5 B3}, we begin with the following decompositon
\begin{equation}\label{eq sec5 B5}
p_{\e}=\int_{|x-p^*_{\e}|\le \e^{\frac12}}xw_{\e}(x)\,dx+\int_{|x-p^*_{\e}|> \e^{\frac12}}xw_{\e}(x)\,dx.
\end{equation}
It follows from \eqref{eq sec4 C5} and \eqref{eq sec5 B4} that (recall that $|p^*_{\e}|\le 100$)
\begin{equation}\begin{aligned}
\int_{|x-p^*_{\e}|\le \e^{\frac12}}xw_{\e}(x)\,dx&= \int_{|x-p^*_{\e}|\le \e^{\frac12}}(x-p^*_{\e})w_{\e}(x)\,dx+p^*_{\e}\int_{|x-p^*_{\e}|\le \e^{\frac12}}w_{\e}(x)\,dx \\
&=O(\e^{\frac12})+p^*_{\e}\left( 1+O\left(\frac{1}{\log \frac{1}{\e}} \right) \right)\\
&=p^*_{\e}+O\left(\frac{1}{\log \frac{1}{\e}} \right)
\end{aligned}\end{equation}
and (again using \eqref{eq sec5 B4})
$$
\left|\int_{|x-p^*_{\e}|> \e^{\frac12}}xw_{\e}(x)\,dx \right| =O\left(\frac{1}{\log \frac{1}{\e}} \right).
$$
Thus, \eqref{eq sec5 B5} gives
$$
p_{\e}=p^*_{\e}+O\left(\frac{1}{\log \frac{1}{\e}} \right),
$$
which proves \eqref{eq sec5 B3}. Next we prove \eqref{eq sec5.2 C8}, the fact $|p^*_{\e}|\le 100$ implies that
\begin{equation}\begin{aligned}
\int_{|x-p^*_{\e}|\ge \e^{1- \alpha}} |x|^2w_{\e}(x)\,dx &\lesssim \int_{|x-p^*_{\e}|\ge 200} |x|^2w_{\e}(x)\,dx +\int_{\e^{1-\alpha}\le|x-p^*_{\e}|\le 200} |x|^2w_{\e}(x)\,dx \\
&\lesssim \int_{|x|\ge 100} |x|^2w_{\e}(x)\,dx +\int_{|x-p^*_{\e}|\ge \e^{1- \alpha}}|x|w_{\e}(x)\,dx.
\end{aligned}\end{equation}
Using \eqref{eq sec5 B4} and our definition of $M^*(t)$ in Section $4$, we get
\begin{equation}\label{eq sec5.2 C6}\begin{aligned}
\int_{|x-p^*_{\e}|\ge \e^{1- \alpha}} |x|^2w_{\e}(x)\,dx &\lesssim M^*(t)+\frac{1}{\alpha\log \frac{1}{\e}}.
\end{aligned}\end{equation}
As a consequence \eqref{eq sec4.1 C9}, \eqref{eq sec5 B3} and Lemma \ref{le sec4.3 C8}, Gronwall's inequality then gives
$$
M^*(t)\lesssim \frac{1}{\log \frac{1}{\e}}.
$$
Together with \eqref{eq sec5.2 C6}, we obtain \eqref{eq sec5.2 C8}, which leads directly to
\begin{equation}\label{eq sec5 E3}\begin{aligned}
\int_{\R^2} |x-p^*_{\e}|^2w_{\e}(x)\,dx&=\int_{|x-p^*_{\e}|\ge \e^{\frac12}} |x-p^*_{\e}|^2w_{\e}(x)\,dx+\int_{|x-p^*_{\e}|< \e^{\frac12}} |x-p^*_{\e}|^2w_{\e}(x)\,dx \\
&=O\left(\frac{1}{\log\frac{1}{\e}}\right)+O(\e)=O\left(\frac{1}{\log\frac{1}{\e}}\right).
\end{aligned}\end{equation}
It is worth noting that the center of gravity $p_{\e}$ is the minimizer of the functional $$F(p):=\int_{\R^2} |x-p|^2 w_{\e}(x)\,dx.$$
Thus, \eqref{eq sec5 E3} yields
$$
\int_{\R^2} |x-p_{\e}|^2w_{\e}(x)\,dx =O\left(\frac{1}{\log\frac{1}{\e}}\right)
$$
and \eqref{eq sec5 C0}, \eqref{eq sec5 C1} follows directly from \eqref{eq sec4 C2}, \eqref{eq sec4 D0} and \eqref{eq sec4 D1}.
\end{proof}

\subsection{Dynamics for $p_{e}(t)$}
We will consider the movement of $p_{\e}(t)$ in this subsection. Our main theorem is stated as follows:
\begin{theorem}\label{thm sec5 main}
For any $t \in [0,T)$, by chooing $\e$ small enough, there holds
\begin{equation*}
p_{\e}(t)=\left( \cos \left(V_{\e}t \right),\sin \left(V_{\e}t \right) \right) +O\left( t  \right),
\end{equation*}
where
$$
V_{\e}=-\frac{\sqrt2}{8\pi}\log\frac{1}{\e}.
$$
\end{theorem}
To prove the theorem, it suffices to show that
\begin{equation}\label{eq sec5.3 A1}
\frac{dp_{\e}}{dt}=-\frac{\sqrt2}{8\pi}\log\frac{1}{\e}p_{\e}^{\perp}+O(1).
\end{equation}
We begin with the following useful lemma,
\begin{lemma}\label{le sec5 A0}
Let $f(x,t)$ be a non-negative function satisfying
$$\int_{\R^2} f(x,t)\,dx <\infty$$
for all $t\in [0,T)$. Furthermore, assuming that $w_{\e}(x,t)$ satisfies Assumption \ref{as initial data}. Then there holds
\begin{equation}\label{eq sec 5.3 A0}
\int_{\R^2}\int_{\R^2} \log \frac{1}{|x-y|}\mathbf{1}_{|x-y|\le 1}f(x,t)w_{\e}(y,t)\,dx \,dy \le C \log\frac{1}{\e}\int_{\R^2}f(x,t)\,dx.
\end{equation}
\end{lemma}
\begin{proof}
By setting $\tilde{w}_{\e}(x,t):=\e^2 w_{\e}(\e x,t)$, a direct calculation shows
\begin{equation*}\begin{aligned}
&\int_{\R^2}\int_{\R^2} \log \frac{1}{|x-y|}\mathbf{1}_{|x-y|\le 1}f(x,t)w_{\e}(y,t)\,dx \,dy \\
=& \int_{\R^2}\int_{\R^2} \left(\log \frac{1}{\e}+ \log \frac{1}{|\frac{x}{\e}-y|} \right)\mathbf{1}_{|x-\e y|\le 1}f(x,t)\tilde{w}_{\e}(y,t)\,dx \,dy \\
\le& \log\frac{1}{\e}\int_{\R^2}\int_{\R^2} f(x,t)\tilde{w}_{\e}(y,t)\,dx \,dy \\
&+\int_{\R^2}\int_{\R^2} \frac{\tilde{w}_{\e}(y,t)}{|\frac{x}{\e}-y|}\,dy\, f(x,t) \,dx \\
\lesssim& \log \frac{1}{\e}\int_{\R^2} f(x,t)\,dx,
\end{aligned}\end{equation*}
where we have used the fact that
$$
\int_{\R^2} \frac{\tilde{w}_{\e}(y,t)}{|\frac{x}{\e}-y|}\,dy \lesssim \|\tilde{w}_{\e}\|_{L^1\bigcap L^{\infty}} \lesssim 1.
$$
\end{proof}
Now we are ready to prove our main theorem
\begin{proof}[Proof of Theorem \ref{thm sec5 main}]
To simplify the notation, we will omit the time variable $t$ and the subscript $\e$ throughout the proof. Recalling that
$$
\frac{dp}{dt}=\int_{\R^2} w(x)Hw(x)\,dx
$$
and it suffices to prove \eqref{eq sec5.3 A1}. Therefore, we use \eqref{eq sec4 C5}, \eqref{eq sec4 C3} and Theorem \ref{thm sec3 A1} to compute that
\begin{equation*}\begin{aligned}
\frac{dp}{dt}=&-\frac{1}{4\pi}\int_{\R^2}\int_{\R^2}\log \frac{1}{|x-y|}\mathbf{1}_{|x-y|\le 1}\frac{x^{\perp}}{\langle x \rangle}  w(x)w(y)\,dx\,dy+ O(1) \\
=&-\frac{1}{4\pi}\int_{\R^2}\int_{\R^2}\log \frac{1}{|x-y|}\mathbf{1}_{|x-y|\le 1}\frac{(x-p)^{\perp}}{\langle x \rangle}  w(x)w(y)\,dx\,dy+O(1) \\
&-\frac{p^{\perp}}{4\pi}\int_{\R^2}\int_{\R^2}\log \frac{1}{|x-y|}\mathbf{1}_{|x-y|\le 1}\frac{1}{\langle x \rangle}  w(x)w(y)\,dx\,dy.
\end{aligned}\end{equation*}
Utilizing \eqref{eq sec4 C5}, \eqref{eq sec5 B4}, \eqref{eq sec5 B3} and \eqref{eq sec5 C0}, we obtain
\begin{equation*}\begin{aligned}
\int_{\R^2} \frac{|x-p|}{\langle x \rangle}  w(x)\,dx \lesssim& \int_{\R^2} \frac{|x-p^*|}{\langle x \rangle}  w(x)\,dx+\int_{\R^2} \frac{|p^*-p|}{\langle x \rangle}  w(x)\,dx \\
\lesssim& \int_{|x-p^*|\le \e^{\frac12}} \frac{|x-p^*|}{\langle x \rangle}  w(x)\,dx+\int_{|x-p^*|> \e^{\frac12}} \frac{|x-p^*|}{\langle x \rangle}  w(x)\,dx \\&+O\left( \frac{1}{\log \frac{1}{\e}} \right) \\
\lesssim& O\left( \frac{1}{\log \frac{1}{\e}} \right),
\end{aligned}\end{equation*}
which, combined with Lemma \ref{le sec5 A0}, implies
$$
\int_{\R^2}\int_{\R^2}\log \frac{1}{|x-y|}\mathbf{1}_{|x-y|\le 1}\frac{(x-p)^{\perp}}{\langle x \rangle}  w(x)w(y)\,dx\,dy=O(1).
$$
Therefore, it remains to show that
\begin{equation*}
\int_{\R^2}\int_{\R^2}\log \frac{1}{|x-y|}\mathbf{1}_{|x-y|\le 1}\frac{1}{\langle x \rangle}  w(x)w(y)\,dx\,dy =\frac{\log\frac{1}{\e}}{\sqrt2}+O\left( 1 \right).
\end{equation*}
To this end, we utilize \eqref{eq sec4 C1} to compute
\begin{equation*}\begin{aligned}
&\int_{\R^2}\int_{\R^2}\log \frac{1}{|x-y|}\mathbf{1}_{|x-y|\le 1}\frac{1}{\langle x \rangle}  w(x)w(y)\,dx\,dy \\
=& \int_{\R^2}\int_{\R^2}\log \frac{1}{|x-y|}\mathbf{1}_{|x-y|\le 1}\frac{1-|x|^2}{2\langle x \rangle}  w(x)w(y)\,dx\,dy +\frac{E^*}{2} \\
=& \frac{\log\frac{1}{\e}}{\sqrt2}+\int_{\R^2}\int_{\R^2}\log \frac{1}{|x-y|}\mathbf{1}_{|x-y|\le 1}\frac{1-|x|^2}{2\langle x \rangle}  w(x)w(y)\,dx\,dy+O(1).
\end{aligned}\end{equation*}
Consequently, to completes the proof, it suffices to prove that (using Lemma \ref{le sec5 A0})
\begin{equation}\label{eq sec5.3 A2}
\int_{\R^2}\bigg| 1-|x|\bigg|  w(x)\,dx=O\left( \left(\log \frac{1}{\e} \right)^{-1} \right).
\end{equation}
So we estimate
\begin{equation*}\begin{aligned}
\int_{\R^2}\bigg| 1-|x|\bigg|  w(x)\,dx\le&\int_{\R^2}\bigg| 1-|p^*|\bigg|  w(x)\,dx+\int_{\R^2}| p^*-x|  w(x)\,dx\\
\le& \int_{\R^2}\bigg| 1-|p^*|\bigg|  w(x)\,dx+\int_{|x-p^*|\ge \e^{\frac12}}| p^*-x|  w(x)\,dx+O(\e^{\frac12}).
\end{aligned}\end{equation*}
Finally, by utilizing \eqref{eq sec4 C5}, \eqref{eq sec5 B4}, \eqref{eq sec5 B3} and \eqref{eq sec5 C0}, we arrive at
\begin{equation*}\begin{aligned}
\int_{\R^2}\bigg| 1-|x|\bigg|  w(x)\,dx=O\left( \left(\log \frac{1}{\e} \right)^{-1} \right),
\end{aligned}\end{equation*}
which proves \eqref{eq sec5.3 A2} and hence completes the proof.
\end{proof}

\end{document}